\documentclass[american,12pt]{amsart}

\usepackage{type1cm} 
\usepackage{geometry}
\usepackage{babel}
\usepackage{graphicx}
\usepackage{cases}
\usepackage{mathtools}
\usepackage{amsthm}
\usepackage{amssymb}
\usepackage{tikz}
\usepackage{xcolor}
\usepackage{dashbox}
\usepackage{skak}
\usepackage{bm}  
\usepackage{verbatim}
\usepackage{enumitem}
\setlist{itemsep=0pt,topsep=0pt,parsep=1pt,partopsep=0pt}

\usepackage[noadjust]{cite}
\usepackage{xurl}

\definecolor{darkblue}{rgb}{0, 0, .6}
\definecolor{grey}{rgb}{.7, .7, .7}

\makeatletter

\numberwithin{equation}{section}
\numberwithin{figure}{section}
\numberwithin{table}{section}
\theoremstyle{definition}
\newtheorem{thm}{\protect\theoremname}[section]
\theoremstyle{definition}
\newtheorem{example}[thm]{\protect\examplename}
\newtheorem{prop}[thm]{\protect\propositionname}
\newtheorem{conjecture}[thm]{\protect\conjecturename}
\newtheorem{rem}[thm]{\protect\remarkname}
\newtheorem{corollary}[thm]{\protect\corollaryname}
\newtheorem{lemma}[thm]{Lemma}

\usepackage{pgf,tikz}
\tikzstyle{vert} = [circle, draw, fill=grey!20,inner sep=0pt, minimum size=5mm]
\tikzstyle{med vert} = [circle, draw, fill=grey!20,inner sep=0pt, minimum size=4mm]
\tikzstyle{small vert} = [circle, draw, fill=grey!20,inner sep=0pt, minimum size=3mm]
\tikzstyle{b} = [draw, very thick, black,-]
\tikzstyle{a} = [draw, black,-stealth]

\usetikzlibrary{arrows,backgrounds,calc,trees}
\usetikzlibrary{arrows.meta,positioning,decorations.pathmorphing}

\pgfdeclarelayer{background}
\pgfsetlayers{background,main}

\DeclareRobustCommand{\cc}[1]{%
\begin{tikzpicture}
\node[draw,circle,red,inner sep=.7] {$\scriptscriptstyle #1$};
\end{tikzpicture}%
}

\newcommand{\convexpath}[2]{
[   
    create hullnodes/.code={
        \global\edef\namelist{#1}
        \foreach [count=\counter] \nodename in \namelist {
            \global\edef\numberofnodes{\counter}
            \node at (\nodename) [draw=none,name=hullnode\counter] {};
        }
        \node at (hullnode\numberofnodes) [name=hullnode0,draw=none] {};
        \pgfmathtruncatemacro\lastnumber{\numberofnodes+1}
        \node at (hullnode1) [name=hullnode\lastnumber,draw=none] {};
    },
    create hullnodes
]
($(hullnode1)!#2!-90:(hullnode0)$)
\foreach [
    evaluate=\currentnode as \previousnode using \currentnode-1,
    evaluate=\currentnode as \nextnode using \currentnode+1
    ] \currentnode in {1,...,\numberofnodes} {
  let
    \p1 = ($(hullnode\currentnode)!#2!-90:(hullnode\previousnode)$),
    \p2 = ($(hullnode\currentnode)!#2!90:(hullnode\nextnode)$),
    \p3 = ($(\p1) - (hullnode\currentnode)$),
    \n1 = {atan2(\y3,\x3)},
    \p4 = ($(\p2) - (hullnode\currentnode)$),
    \n2 = {atan2(\y4,\x4)},
    \n{delta} = {-Mod(\n1-\n2,360)}
  in 
    {-- (\p1) arc[start angle=\n1, delta angle=\n{delta}, radius=#2] -- (\p2)}
}
-- cycle
}

\tikzset{
    my box/.style = {
        , line cap = round
        , line join = round
    }
}

\usetikzlibrary{arrows}
\usetikzlibrary{shapes}
\usetikzlibrary{patterns}

\renewcommand*\env@cases[1][1]{%
  \let\@ifnextchar\new@ifnextchar
  \left\lbrace
  \def\arraystretch{#1}%
  \array{@{}l@{\quad}l@{}}%
}

\makeatother

\providecommand{\conjecturename}{Conjecture}
\providecommand{\examplename}{Example}
\providecommand{\propositionname}{Proposition}
\providecommand{\remarkname}{Remark}
\providecommand{\theoremname}{Theorem}
\providecommand{\corollaryname}{Corollary}

\global\long\def\TER{\text{TER}}%
\global\long\def\DNT{\text{DNT}}%
\global\long\def\dawsonschess{\text{DC}}%
\global\long\def\SPL{\text{SPL}}%

\DeclareMathOperator{\Wd}{Wd}
\global\long\def\Opt{\text{Opt}}%
\global\long\def\mex{\text{mex}}%
\global\long\def\nim{\text{nim}}%
\global\long\def\pty{\text{pty}}%
\global\long\def\Tr{\mathrm{Tr}}%

\newcommand{\m}[1]{\{\!\!\{ #1 \}\!\!\}}

\tikzstyle{small vert} = [circle, draw, fill=grey!60,inner sep=0pt, minimum size=3mm]
\tikzstyle{b} = [draw, very thick, black,-]
\tikzstyle{a} = [draw, black,-stealth]

\begin{document}

\title[Impartial Geodetic Removing Games on Graphs]{Impartial Geodetic Removing Games on Graphs}

\author{Bret J.~Benesh}
\address{
Department of Mathematics,
College of Saint Benedict and Saint John's University,
37 College Avenue South,
Saint Joseph, MN 56374-5011, USA
}
\email{bbenesh@csbsju.edu}

\author{Dana C.~Ernst}
\address{
Department of Mathematics and Statistics,
Northern Arizona University PO Box 5717,
Flag\-staff, AZ 86011-5717, USA
}
\email{Dana.Ernst@nau.edu, Nandor.Sieben@nau.edu}

\author{Marie Meyer}
\address{
Department of Mathematics and Statistics,
University of South Florida,
4202 E. Fowler Avenue,
Tampa, FL 33620, USA
}
\email{mariemeyer@usf.edu}

\author{Sarah K.~Salmon}
\address{
Department of Mathematics,
University of Colorado Boulder
Campus Box 395,
2300 Colorado Avenue,
Boulder, CO 80309, USA}
\email{Sarah.Salmon@colorado.edu}

\author{N\'andor Sieben}

\thanks{Date: \the\month/\the\day/\the\year}



\begin{abstract}
A subset of the vertex set of a graph is geodetically convex if it contains every vertex on any shortest path between two elements of the subset. The convex hull of a set of vertices is the smallest convex set containing the set. 
We study two games in which two players take turns selecting vertices of a graph until the convex hull of the jointly unselected vertices is too small. The last player to move is the winner. The achievement game ends when the convex hull of the jointly unselected vertices is not the vertex set. In the avoidance game, the convex hull of the jointly unselected vertices must always be the vertex set.  We study the nim-values for several graph families, including cycle graphs, hypercube graphs, complete multipartite graphs, wheel graphs, generalized wheel graphs, and graphs with a unique minimal generating set. 
\end{abstract}

\maketitle

\section{Introduction}

Given a subset of vertices $S$ in a finite simple graph, the set of vertices lying on any shortest path between elements of $S$ is known as the geodetic closure of $S$. This concept formed the basis for a pair of impartial combinatorial games introduced by Harary~\cite{HARARY1984323}. In these achievement and avoidance building games, two players alternately select unselected vertices, updating the geodetic closure of their joint selections at each turn. 
The achievement game ends as soon as the geodetic closure becomes the vertex set, whereas the avoidance game does not allow the geodetic closure to be equal to the vertex set. The player who is unable to move loses the games.
Variants of geodetic closure building games have been extensively studied across standard graph families, including cycles, wheels, complete multipartite graphs, and split graphs~\cite{BuckleyHarary85,BuckleyHarary86,FraenkelHarary89,HaynesHenningTiller,Necascova,wang17}.

In~\cite{BEMSS2024}, we introduced a variation of the geodetic closure games that uses the convex hull. A vertex subset $S$ is geodetically convex if it contains all vertices along shortest paths between its members. The convex hull of $S$ is the smallest geodetically convex set containing $S$. Despite its name, the geodetic closure is only a pre-closure, while the convex hull is a closure operator. Although the convex hull and geodetic closure operators are identical for various graph families, they diverge on others. While most of the results in the literature have restricted their focus to the outcomes of geodetic closure games, we determined the nim-values of convex hull games across multiple graph families.

The geodetic closure and convex hull games described above are examples of building hypergraph games, which were formalized in~\cite{SiebenHypergraph}. In contrast, removing hypergraph games represent a complementary variation of hypergraph games. We initiated a study of removing games played on graphs using the convex hull operator in~\cite{impartialremovinggamesgrid}. For the removing paradigm, players select vertices from a finite graph, steadily shrinking the pool of unchosen vertices until their convex hull becomes too small. Specifically, the achievement game $\TER$ (terminate) is won by the player who is able to select a vertex such that the convex hull of the unselected vertices no longer equals the vertex set.  The avoidance game $\DNT$ (do not terminate) requires the players to keep the convex hull of the unselected vertices equal to the vertex set. In~\cite{impartialremovinggamesgrid}, we determined the nim-value of these games for the family of grid graphs and also provided some results for higher-dimensional lattice graphs. 

In the present paper, we compute nim-values for $\TER$ and $\DNT$ across an array of graph classes, such as complete split, corona, block, cycle, hypercube, complete multipartite, wheel, and generalized wheel graphs. To accomplish this, our analysis employs several combinatorial strategies, including characterizing maximal nonterminating sets, establishing structural equivalences via option-preserving maps, reducing positions to the game of Dawson's Chess, and using case analysis diagrams to describe complex winning strategies.

For a foundational treatment of impartial game theory, readers are directed to~\cite{albert2007lessons,ONAG,SiegelBook}. The remainder of this paper is organized as follows. Section~\ref{sec:Preliminaries} establishes preliminary definitions regarding transversals, impartial games, and geodetic convexity. Section~\ref{sec:RemovingGames} formally defines the $\TER$ and $\DNT$ removing games. In Section~\ref{sec:UniqueMinGenSet}, we resolve the nim-values for graphs possessing a unique minimal generating set. Section~\ref{sec:GraphFamilies} details our nim-value computations for several graph families. Finally, Section~\ref{sec:Closing} highlights open questions and potential avenues for future research.

\section{Preliminaries\label{sec:Preliminaries}}

We recall some terminology and notation.

\subsection{Notation}
If $f:X\to Y$ and $A\subseteq X$, then we often use the standard
$f(A):=\{f(a)\mid a\in A\}$ notation for the image of $A$. 
As a special case, for a family $\mathcal{A}$ of subsets of $X$ we define
\[
\complement(\mathcal{A}):=\{A^c\mid A\in\mathcal{A}\},
\]
where $A^c=X\setminus A$ is the complement of $A$.

We define the \emph{parity of the integer} $k$ as $\pty(k):=k\!\mod2$. The cardinality of a set $A$ is denoted by $|A|$, and we write $\pty(A):=\pty(|A|)$
for the \emph{parity of a set}.

\subsection{Transversals}

Let $\mathcal{A}$ be a family of sets.  A set $T$ is a \emph{transversal of $\mathcal{A}$} if $T \cap A \not= \emptyset$ for all $A \in \mathcal{A}$. We define $\Tr(\mathcal{A})$ to be the set of minimal transversals of $\mathcal{A}$. Transversals are sometimes called blocking sets, in which case $\Tr(\mathcal{A})$ is called the \emph{blocker} of $\mathcal{A}$.

\begin{example} 
One can verify that 
\[
\Tr(\left\{\{1,2\}, \{2,3,4\}, \{2,3,5\}\right\})=\{\{2\},\{1,3\},\{1,4,5\}\}.
\]
The special cases $\Tr( \emptyset ) = \{ \emptyset\}$ and $\Tr( \{ \emptyset \} )= \emptyset$ play important roles.
\end{example}

A family $\mathcal{A}$ of sets is a \emph{Sperner family} or \emph{clutter} if no element of $\mathcal{A}$ contains another element of $\mathcal{A}$. If $\mathcal{A}$ is a Sperner family, then $\Tr(\Tr(\mathcal{A}))=\mathcal{A}$.

\subsection{Impartial games}

Our general references for combinatorial game theory are~\cite{albert2007lessons,SiegelBook}.

In an \emph{impartial game}, two players take turns to replace the current position of the game with one of its options. The game ends when the current position has no options. The player unable to move is the \emph{loser}. Every game must finish in finitely many steps. In particular, no position can be reached twice. We model an impartial game with a \emph{gamegraph}, which is a finite set $\mathsf{G}$ of \emph{positions} and a collection
$\Opt(p)\subseteq\mathsf{G}$ of \emph{options} for each position
$p\in\mathsf{G}$. A gamegraph has a \emph{starting position}, which is a unique position not in the option set of any position.
We visualize gamegraphs with a diagram showing an arrow from a position to every option of that position. Game play is moving from one vertex to another along the arrows. The game ends when a position without options is reached. 

The \emph{minimum excludant} $\mex(A)$ of a set $A$ of non-negative integers is the smallest non-negative integer that is not in $A$. The \emph{nim-value} $\nim(p)$ of a position $p$ is defined recursively as the minimum excludant of the nim-values of the options of $p$. That is,
\[
\nim(p):=\mex(\nim(\Opt(p))).
\]
The \emph{nim-value of the game} is the nim-value of the starting position.

A position is \emph{terminal} if it has no options. A terminal position $p$ has nim-value $\nim(p)=\mex(\nim(\emptyset))=\mex(\emptyset)=0$.
A position $p$ is \emph{losing} for the player about to move ($P$-position) if $\nim(p)=0$ and \emph{winning} ($N$-position) otherwise. The winning strategy is to always move to a losing option with nim-value $0$ if available. 

The \emph{game sum} $\mathsf{G}+\mathsf{H}$ has position set $\mathsf{G}\times\mathsf{H}$ with $\Opt(p,q):=(\Opt(p)\times\{q\})\cup(\{p\}\times\Opt(q))$. A convenient way to show that $\nim(\mathsf{G})=k$ is to find a strategy for the second player to win $\mathsf{G}+*k$, where $*k$ is the nimber with $\Opt(*k)=\{*0,\ldots,*(k-1)\}$. 

We will often use another useful result.

\begin{prop}
\label{prop:all-same general}
\cite[Proposition 2.2]{BEMSS2024}
If every terminal position of an impartial game $\mathsf{G}$ has the same parity $r$, then $\nim(\mathsf{G})=r$.
\end{prop}

\subsection{Option-preserving maps}
A function $f:\mathsf{G}\to\mathsf{H}$ between two gamegraphs is \emph{ option-preserving} \cite{Baltushkin,Basic} if $\Opt(f(p))=f(\Opt(p))$ for all $p\in\mathsf{G}$. We use option-preserving maps to study a complicated game through its simpler image. This is possible because of the following result. 

\begin{prop}
\cite[Proposition 4.17]{Basic}
If $f:\mathsf{G}\to\mathsf{H}$ is option-preserving, then $\nim(f(p))=\nim(p)$ for all $p\in\mathsf{G}$.
\end{prop}

This implies that if the starting position of $\mathsf{G}$ maps to the starting position of $\mathsf{H}$, then $\nim(\mathsf{G})=\nim(\mathsf{H})$. Note that a surjective option-preserving map always takes starting positions to starting positions.

\subsection{Geodetic convexity}

A graph is an ordered pair $G=(V,E)$, where $V$ is a finite nonempty set of vertices and $E\subseteq 2^V$ is the set of edges. We do not allow loop edges. A \emph{geodesic} of a graph is a shortest path between two vertices.  A set $P$ of vertices of a graph $(V,E)$ is called \emph{geodetically convex} or simply \emph{convex} if it contains all vertices along the geodesics connecting two vertices of $P$. The \emph{convex hull} $[P]:=\bigcap\{K\mid P\subseteq K, K \text{ is convex} \}$ of $P$ is the smallest convex set containing $P$. The convex hull function $P\mapsto[P]:2^V\to 2^V$ is a closure operator. In particular, $[P]$ is convex if and only if $[P]=P$. 
A comprehensive reference about geodetic convexity is \cite{PelayoBook}.

We say that a set $P$ of vertices is \emph{generating} if $[P]=V$. Otherwise, $P$ is called \emph{nongenerating}.
The family of maximal nongenerating sets is denoted by $\mathcal{N}$ while the family of minimal generating sets is denoted by $\mathcal{G}$. 

We say that a set $P$ of vertices is \emph{terminating} if $[P^c]\ne V$. Otherwise, $P$ is called \emph{nonterminating}. The family of maximal nonterminating sets is denoted by $\mathcal{N}^\star$ while the family of minimal terminating sets is denoted by $\mathcal{G}^\star$. 
The generating and terminating sets are Sperner families that are related according to the following result from~\cite[Subsection 2.2]{SiebenHypergraph}.

\begin{prop}
For all graphs, $\mathcal{N}^\star=\complement(\mathcal{G})$, $\mathcal{N}=\complement(\mathcal{G}^\star)$, and $\mathcal{G}^*=\Tr(\mathcal{G})$.
\end{prop}

\begin{rem}\label{remark:transversal complement voodoo}
The relationships can be summarized by the following diagram:
\[
\mathcal{N}^\star
\overset{\complement}{\longleftrightarrow} 
\mathcal{G}
\overset{\Tr}{\longleftrightarrow}
\mathcal{G}^\star
\overset{\complement}{\longleftrightarrow}
\mathcal{N}
\]
\end{rem}

\begin{example}\label{ex:kite example}
The kite graph $G$ with its generating and terminating sets is shown in Figure~\ref{fig:kite}. 
\end{example}

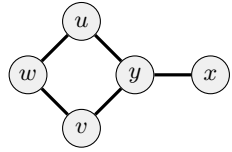
\begin{figure}[h]

\begin{tikzpicture}[scale=1,auto,rotate=45,baseline=5mm]
\node (4) at (0,0) [ vert] {\scriptsize $v$};
\node (2) at (1,0) [ vert 
] {\scriptsize $y$};
\node (3) at (1,1) [ vert] {\scriptsize $u$};
\node (5) at (0,1) [ vert] {\scriptsize $w$};
\node (1) at (1.707106,-0.707106) [ vert] {\scriptsize $x$};
\path [b] (4) to (2) to (3) to (5) to (4);
\path [b] (1) to (2);
\end{tikzpicture} 
$\qquad
\begin{aligned}
\mathcal{N}   &= \{\{x,y,u\},\{x,y,v\},\{y,u,v,w\}\} \\
\mathcal{G}^* &= \{\{v,w\},\{u,w\},\{x\}\} \\
\mathcal{G}   &= \{ \{x,w\},\{x,u,v\} \} \\
\mathcal{N}^* &= \{\{y,u,v\},\{y,w\}\} 
\end{aligned}
$

\caption{
\label{fig:kite}
A graph with its generating and terminating sets. 
}
\end{figure}

\section{Removing games}\label{sec:RemovingGames}

Our goal is to study two \emph{impartial removing hypergraph games} \cite{SiebenHypergraph}. We play both games on a graph $G$ with vertex set $V$ and edge set $E$. Two players take turns selecting previously unselected vertices of $G$ until certain conditions are met. Each game position is the set $P$ of jointly selected vertices.

The achievement game \emph{terminate} $\TER(G)$ ends as soon as the set of unchosen vertices of $G$ no longer generates $V$. So the player who first removes a vertex that prevents generating the whole vertex set wins. This happens as soon as $P$ is terminating. That is, $G\subseteq P$ for some $G\in\mathcal{G}^*$.

In the avoidance game \emph{do not terminate} $\DNT(G)$, each position $P$ must be nonterminating. That is, $P\subseteq N$ for some $N\in\mathcal{N}^*$. The game ends if no additional vertex can be selected while maintaining this condition. 

\begin{example}
The gamegraphs for the path graph $P_3$ with $V=\{u,v,w\}$ are shown in Figure~\ref{fig:P3}. Note that $\mathcal{G}^*=\{\{u\},\{w\}\}$ and $\mathcal{N}^\star=\{\{v\}\}$.
\end{example}

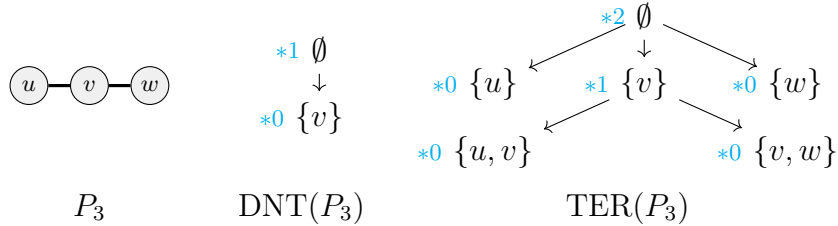
\begin{figure}[h!]

\begin{tabular}{cccc}
\begin{tabular}{c}
\begin{tikzpicture}[scale=.8,auto]
\node (u) at (1,0) [ vert ] {\scriptsize $u$};
\node (v) at (2,0) [ vert] {\scriptsize $v$};
\node (w) at (3,0) [ vert] {\scriptsize $w$};
\path [b] (u) to (v) to (w) ;
\end{tikzpicture} 
\end{tabular} & %
 & %
\begin{tabular}{c}
\begin{tikzpicture}[yscale=.9]
\node[label={[label distance=-4pt]left:{\scriptsize $\color{cyan}{*1}$}}] (0) at (1,3) {$\emptyset$};
\node[label={[label distance=-4pt]left:{\scriptsize $\color{cyan}{*0}$}}] (1) at (1,2) {$\{v\}$};
\draw[->] (0) -- (1);
\end{tikzpicture}
\end{tabular} & %
\begin{tabular}{c}
\begin{tikzpicture}[xscale=2,yscale=.9]
\node[label={[label distance=-4pt]left:{\scriptsize $\color{cyan}{*2}$}}] (0) at (1,3) {$\emptyset$};
\node[label={[label distance=-4pt]left:{\scriptsize $\color{cyan}{*1}$}}] (00) at (1,2) {$\{v\}$};
\node[label={[label distance=-4pt]left:{\scriptsize $\color{cyan}{*0}$}}] 
(1) at (0,2) {$\{u\}$}; 
\node[label={[label distance=-4pt]left:{\scriptsize $\color{cyan}{*0}$}}] (2) at (2,2) {$\{w\}$}; 
\node[label={[label distance=-4pt]left:{\scriptsize $\color{cyan}{*0}$}}] (3) at (0,1) {$\{u,v\}$}; 
\node[label={[label distance=-4pt]left:{\scriptsize $\color{cyan}{*0}$}}] (4) at (2,1) {$\{v,w\}$}; 
\draw[->] (0) -- (00);
\draw[->] (0) -- (1);
\draw[->] (0) -- (2);
\draw[->] (00) -- (3);
\draw[->] (00) -- (4);
\end{tikzpicture}
\end{tabular}\\
$P_3$ &  & $\DNT(P_3)$ & $\TER(P_3)$ 
\end{tabular}

\caption{
\label{fig:P3} 
Gamegraphs with nim-values for $P_3$.
}
\end{figure}

\begin{example}
\label{ex:single vertex}
Let $G$ be the graph with a single vertex $v$, so that $\mathcal{G}^*=\{\{v\}\}$ and $\mathcal{N}^*=\{\emptyset\}$. The nim-value of $\DNT(G)$ is $0$ since the only position of the game is $\emptyset$.
The nim-value of $\TER(G)$ is $1$ since the game has only two positions $\emptyset$ and $\{v\}$. 
\end{example}

The terminal positions of $\DNT(G)$ are the elements of $\mathcal{N}^*$. Hence we have the following consequence of Proposition~\ref{prop:all-same general}.

\begin{prop}\label{prop:all-same}
If every element of $\mathcal{N}^*$ has the same parity $r$, then $\nim(\DNT(G))=r$.
\end{prop}

\begin{example}
\label{ex:W5}
Consider the wheel graph $W_5$. Representative quotient gamegraphs for $\DNT(W_5)$ and $\TER(W_5)$ are given in Figure~\ref{fig:W5}. In this quotient, we identified geometrically congruent positions. The canonical quotient map is option-preserving. In both cases, we have labeled positions with their corresponding nim-values. Every position of $\DNT(W_5)$ contains two unmarked antipodal noncentral vertices that generate $W_5$. The terminal positions of $\TER(W_5)$ do not have such pairs of unmarked vertices.
\end{example}

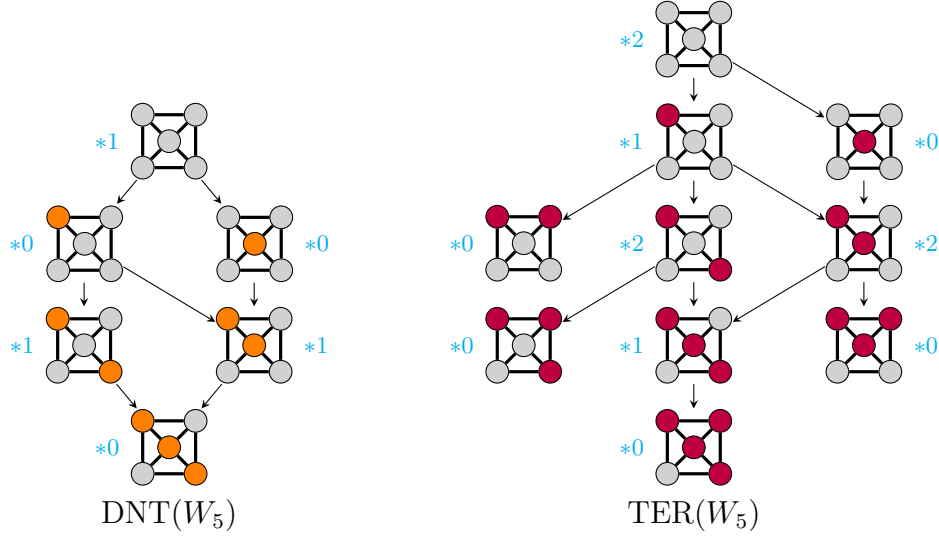
\begin{figure}[h]
\centering
\begin{tabular}{cc}
\begin{tikzpicture}[yscale=.6]
\node[inner sep=0pt,label=left:{\scriptsize $\color{cyan}{*1}$}] (top) at (3.125,5.25) {
\begin{tikzpicture}[scale=.7,auto]
\node (1) at (0,0) [small vert] {};
\node (2) at (1,0) [small vert] {};
\node (3) at (1,1) [small vert] {};
\node (4) at (0,1) [small vert] {};
\node (5) at (.5,.5) [small vert] {};
\path [b] (1) to (2) to (3) to (4) to (1);
\path [b] (1) to (5);
\path [b] (2) to (5);
\path [b] (3) to (5);
\path [b] (4) to (5);
\end{tikzpicture}
};

\node[inner sep=0pt,label=left:{\scriptsize $\color{cyan}{*0}$}] (left1) at (2,3) {
\begin{tikzpicture}[scale=.7,auto]
\node (1) at (0,0) [small vert] {};
\node (2) at (1,0) [small vert] {};
\node (3) at (1,1) [small vert] {};
\node (4) at (0,1) [small vert,fill=orange] {};
\node (5) at (.5,.5) [small vert] {};
\path [b] (1) to (2) to (3) to (4) to (1);
\path [b] (1) to (5);
\path [b] (2) to (5);
\path [b] (3) to (5);
\path [b] (4) to (5);
\begin{pgfonlayer}{background}
\fill[white,opacity=0.3] \convexpath{4,4}{7pt};
\end{pgfonlayer}
\end{tikzpicture}
};
\node[inner sep=0pt,label=right:{\scriptsize $\color{cyan}{*0}$}] (right1) at (4.25,3) {
\begin{tikzpicture}[scale=.7,auto]
\node (1) at (0,0) [small vert] {};
\node (2) at (1,0) [small vert] {};
\node (3) at (1,1) [small vert] {};
\node (4) at (0,1) [small vert] {};
\node (5) at (.5,.5) [small vert,fill=orange] {};
\path [b] (1) to (2) to (3) to (4) to (1);
\path [b] (1) to (5);
\path [b] (2) to (5);
\path [b] (3) to (5);
\path [b] (4) to (5);
\begin{pgfonlayer}{background}
\fill[white,opacity=0.3] \convexpath{4,4}{7pt};
\end{pgfonlayer}
\end{tikzpicture}
};
\node[inner sep=0pt,label=left:{\scriptsize $\color{cyan}{*1}$}] (left2) at (2,.75) {
\begin{tikzpicture}[scale=.7,auto]
\node (1) at (0,0) [small vert] {};
\node (2) at (1,0) [small vert,fill=orange] {};
\node (3) at (1,1) [small vert] {};
\node (4) at (0,1) [small vert,fill=orange] {};
\node (5) at (.5,.5) [small vert] {};
\path [b] (1) to (2) to (3) to (4) to (1);
\path [b] (1) to (5);
\path [b] (2) to (5);
\path [b] (3) to (5);
\path [b] (4) to (5);
\begin{pgfonlayer}{background}
\end{pgfonlayer}
\end{tikzpicture}
};
\node[inner sep=0pt,label=right:{\scriptsize $\color{cyan}{*1}$}] (right2) at (4.25,.75) {
\begin{tikzpicture}[scale=.7,auto]
\node (1) at (0,0) [small vert] {};
\node (2) at (1,0) [small vert] {};
\node (3) at (1,1) [small vert] {};
\node (4) at (0,1) [small vert,fill=orange] {};
\node (5) at (.5,.5) [small vert,fill=orange] {};
\path [b] (1) to (2) to (3) to (4) to (1);
\path [b] (1) to (5);
\path [b] (2) to (5);
\path [b] (3) to (5);
\path [b] (4) to (5);
\begin{pgfonlayer}{background}
\end{pgfonlayer}
\end{tikzpicture}
};
\node[inner sep=0pt,label=left:{\scriptsize $\color{cyan}{*0}$}] (left3) at (3.125,-1.5) {
\begin{tikzpicture}[scale=.7,auto]
\node (1) at (0,0) [small vert] {};
\node (2) at (1,0) [small vert,fill=orange] {};
\node (3) at (1,1) [small vert] {};
\node (4) at (0,1) [small vert,fill=orange] {};
\node (5) at (.5,.5) [small vert,fill=orange] {};
\path [b] (1) to (2) to (3) to (4) to (1);
\path [b] (1) to (5);
\path [b] (2) to (5);
\path [b] (3) to (5);
\path [b] (4) to (5);
\begin{pgfonlayer}{background}
\end{pgfonlayer}
\end{tikzpicture}
};
\path [a] (top) to (left1);
\path [a] (top) to (right1);
\path [a] (left1) to (left2);
\path [a] (right1) to (right2);
\path [a] (left1) to (right2);
\path [a] (left2) to (left3);
\path [a] (right2) to (left3);
\end{tikzpicture} 
\qquad & \qquad
\begin{tikzpicture}[rotate=-90,xscale=.6]
\node[inner sep=0pt,label=left:{\scriptsize $\color{cyan}{*2}$}] (middle1) at (0,0) {
\begin{tikzpicture}[scale=.7,auto]
\node (1) at (0,0) [small vert] {};
\node (2) at (1,0) [small vert] {};
\node (3) at (1,1) [small vert] {};
\node (4) at (0,1) [small vert] {};
\node (5) at (.5,.5) [small vert] {};
\path [b] (1) to (2) to (3) to (4) to (1);
\path [b] (1) to (5);
\path [b] (2) to (5);
\path [b] (3) to (5);
\path [b] (4) to (5);
\end{tikzpicture}
};
\node[inner sep=0pt,label=left:{\scriptsize $\color{cyan}{*1}$}] (middle2) at (2.25,0) {
\begin{tikzpicture}[scale=.7,auto]
\node (1) at (0,0) [small vert] {};
\node (2) at (1,0) [small vert] {};
\node (3) at (1,1) [small vert] {};
\node (4) at (0,1) [small vert,fill=purple] {};
\node (5) at (.5,.5) [small vert] {};
\path [b] (1) to (2) to (3) to (4) to (1);
\path [b] (1) to (5);
\path [b] (2) to (5);
\path [b] (3) to (5);
\path [b] (4) to (5);
\begin{pgfonlayer}{background}
\fill[white,opacity=0.3] \convexpath{4,4}{7pt};
\end{pgfonlayer}
\end{tikzpicture}
};
\node[inner sep=0pt,label=right:{\scriptsize $\color{cyan}{*0}$}] (right2) at (2.25,2.25) {
\begin{tikzpicture}[scale=.7,auto]
\node (1) at (0,0) [small vert] {};
\node (2) at (1,0) [small vert] {};
\node (3) at (1,1) [small vert] {};
\node (4) at (0,1) [small vert] {};
\node (5) at (.5,.5) [small vert,fill=purple] {};
\path [b] (1) to (2) to (3) to (4) to (1);
\path [b] (1) to (5);
\path [b] (2) to (5);
\path [b] (3) to (5);
\path [b] (4) to (5);
\begin{pgfonlayer}{background}
\fill[white,opacity=0.3] \convexpath{4,4}{7pt}; 
\end{pgfonlayer}
\end{tikzpicture}
};
\node[inner sep=0pt,label=left:{\scriptsize $\color{cyan}{*0}$}] (left3) at (4.5,-2.25) {
\begin{tikzpicture}[scale=.7,auto]
\node (1) at (0,0) [small vert] {};
\node (2) at (1,0) [small vert] {};
\node (3) at (1,1) [small vert,fill=purple] {};
\node (4) at (0,1) [small vert,fill=purple] {};
\node (5) at (.5,.5) [small vert] {};
\path [b] (1) to (2) to (3) to (4) to (1);
\path [b] (1) to (5);
\path [b] (2) to (5);
\path [b] (3) to (5);
\path [b] (4) to (5);
\begin{pgfonlayer}{background}
\end{pgfonlayer}
\end{tikzpicture}
};
\node[inner sep=0pt,label=left:{\scriptsize $\color{cyan}{*2}$}] (middle3) at (4.5,0) {
\begin{tikzpicture}[scale=.7,auto]
\node (1) at (0,0) [small vert] {};
\node (2) at (1,0) [small vert,fill=purple] {};
\node (3) at (1,1) [small vert] {};
\node (4) at (0,1) [small vert,fill=purple] {};
\node (5) at (.5,.5) [small vert] {};
\path [b] (1) to (2) to (3) to (4) to (1);
\path [b] (1) to (5);
\path [b] (2) to (5);
\path [b] (3) to (5);
\path [b] (4) to (5);
\begin{pgfonlayer}{background}
\end{pgfonlayer}
\end{tikzpicture}
};
\node[inner sep=0pt,label=right:{\scriptsize $\color{cyan}{*2}$}] (right3) at (4.5,2.25) {
\begin{tikzpicture}[scale=.7,auto]
\node (1) at (0,0) [small vert] {};
\node (2) at (1,0) [small vert] {};
\node (3) at (1,1) [small vert] {};
\node (4) at (0,1) [small vert,fill=purple] {};
\node (5) at (.5,.5) [small vert,fill=purple] {};
\path [b] (1) to (2) to (3) to (4) to (1);
\path [b] (1) to (5);
\path [b] (2) to (5);
\path [b] (3) to (5);
\path [b] (4) to (5);
\begin{pgfonlayer}{background}
\end{pgfonlayer}
\end{tikzpicture}
};
\node[inner sep=0pt,label=left:{\scriptsize $\color{cyan}{*1}$}] (middle4) at (6.75,0) {
\begin{tikzpicture}[scale=.7,auto]
\node (1) at (0,0) [small vert] {};
\node (2) at (1,0) [small vert,fill=purple] {};
\node (3) at (1,1) [small vert] {};
\node (4) at (0,1) [small vert,fill=purple] {};
\node (5) at (.5,.5) [small vert,fill=purple] {};
\path [b] (1) to (2) to (3) to (4) to (1);
\path [b] (1) to (5);
\path [b] (2) to (5);
\path [b] (3) to (5);
\path [b] (4) to (5);
\begin{pgfonlayer}{background}
\end{pgfonlayer}
\end{tikzpicture}
};
\node[inner sep=0pt,label=left:{\scriptsize $\color{cyan}{*0}$}] (left4) at (6.75,-2.25) {
\begin{tikzpicture}[scale=.7,auto]
\node (1) at (0,0) [small vert] {};
\node (2) at (1,0) [small vert,fill=purple] {};
\node (3) at (1,1) [small vert,fill=purple] {};
\node (4) at (0,1) [small vert,fill=purple] {};
\node (5) at (.5,.5) [small vert] {};
\path [b] (1) to (2) to (3) to (4) to (1);
\path [b] (1) to (5);
\path [b] (2) to (5);
\path [b] (3) to (5);
\path [b] (4) to (5);
\begin{pgfonlayer}{background}
\end{pgfonlayer}
\end{tikzpicture}
};
\node[inner sep=0pt,label=right:{\scriptsize $\color{cyan}{*0}$}] (right4) at (6.75,2.25) {
\begin{tikzpicture}[scale=.7,auto]
\node (1) at (0,0) [small vert] {};
\node (2) at (1,0) [small vert] {};
\node (3) at (1,1) [small vert,fill=purple] {};
\node (4) at (0,1) [small vert,fill=purple] {};
\node (5) at (.5,.5) [small vert,fill=purple] {};
\path [b] (1) to (2) to (3) to (4) to (1);
\path [b] (1) to (5);
\path [b] (2) to (5);
\path [b] (3) to (5);
\path [b] (4) to (5);
\begin{pgfonlayer}{background}
\end{pgfonlayer}
\end{tikzpicture}
};
\node[inner sep=0pt,label=left:{\scriptsize $\color{cyan}{*0}$}] (middle5) at (9,0) {
\begin{tikzpicture}[scale=.7,auto]
\node (1) at (0,0) [small vert] {};
\node (2) at (1,0) [small vert,fill=purple] {};
\node (3) at (1,1) [small vert,fill=purple] {};
\node (4) at (0,1) [small vert,fill=purple] {};
\node (5) at (.5,.5) [small vert,fill=purple] {};
\path [b] (1) to (2) to (3) to (4) to (1);
\path [b] (1) to (5);
\path [b] (2) to (5);
\path [b] (3) to (5);
\path [b] (4) to (5);
\begin{pgfonlayer}{background}
\end{pgfonlayer}
\end{tikzpicture}
};
\path [a] (middle1) to (middle2);
\path [a] (middle1) to (right2);
\path [a] (middle2) to (left3);
\path [a] (middle2) to (middle3);
\path [a] (middle2) to (right3);
\path [a] (right2) to (right3);
\path [a] (middle3) to (left4);
\path [a] (middle3) to (middle4);
\path [a] (middle4) to (middle5);
\path [a] (right3) to (right4);
\path [a] (right3) to (middle4);
\end{tikzpicture}\\
$\DNT(W_5)$ \qquad & \qquad $\TER(W_5)$
\end{tabular}

\caption{
\label{fig:W5}
Representative quotients of $\DNT(W_5)$ and $\TER(W_5)$. 
Note that the removal of the terminal positions from $\TER(W_5)$ produces $\DNT(W_5)$.}
\end{figure}

\section{Graphs with a unique minimal generating set}\label{sec:UniqueMinGenSet}

Graphs with a unique minimal generating set are common and relatively easy to analyze.
A vertex is called \emph{simplicial} if the subgraph induced by the neighbors of the vertex is a complete graph. A generating set contains every simplicial vertex by \cite[Proposition 2.7]{BEMSS2024}. Furthermore, if $L$ is a generating set that contains only simplicial vertices, then $\mathcal{G}=\{L\}$ by \cite[Proposition 6.2]{BEMSS2024}.

\begin{example}
Figure~\ref{fig:ugs} shows a graph with a generating set $L=\{u,v,w\}$ consisting of the simplicial vertices. Hence $\mathcal{G}=\{L\}$.
\end{example}

\begin{figure}[h]

\begin{tikzpicture}[scale=.8,auto]
\node (3) at (1,.5) [vert  
] {\scriptsize $x$};
\node (1) at (0,0) [vert] {\scriptsize $u$};
\node (2) at (0,1) [vert] {\scriptsize $v$};
\node (4) at (2,.5) [vert] {\scriptsize $w$};
\path [b] (4) to (3) to (1) to (2) to (3);
\end{tikzpicture}

\caption{
\label{fig:ugs}
A graph with a unique minimal generating set $L=\{u,v,w\}$.
}
\end{figure}
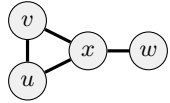

The next result follows easily by either using the definitions of generating and terminating or by utilizing the approach outlined in Remark~\ref{remark:transversal complement voodoo}. 

\begin{prop}\label{prop:unique G}
If $\mathcal{G}=\{L\}$, then $\mathcal{G}^\star=\{\{l\}\mid l\in L\}$ and $\mathcal{N}^\star=\{L^c\}$.
\end{prop}

The following is a consequence of Proposition~\ref{prop:all-same}.

\begin{prop}\label{prop:unique DNT}
If $\mathcal{G}=\{L\}$, then $\nim(\DNT(G))=\pty(L^c)$.
\end{prop}

Using our notation, the next result is a special case of~\cite[Proposition 7.5]{SiebenHypergraph}.

\begin{prop}
If $\mathcal{N}^*$ is pairwise disjoint, $\{\pty(P) \mid  P\in\mathcal{N}^*\}=\{a\}$, and $V(G)\ne\bigcup\mathcal{N}^*$, then $\nim(\TER(G))=a+1$.
\end{prop}

This together with Proposition~\ref{prop:unique G} immediately implies the following.

\begin{prop}\label{prop:unique TER}
If $\mathcal{G}=\{L\}$, then $\nim(\TER(G))=1+\pty(L^c)$.
\end{prop}

\subsection{Complete split graphs}

A \emph{complete split graph} is the join $K_m +\overline{K_n}$ of the complete graph $K_m$ and the complement graph $\overline{K_n}$. Recall from \cite[Proposition~6.8]{BEMSS2024} that $\mathcal{G}=\{V(\overline{K}_n)\}$ for $n\geq 2$. So we have the following by Propositions~\ref{prop:unique DNT} and \ref{prop:unique TER}.

\begin{prop}
\label{prop:splitTER}
If $G=K_m+\overline{K}_n$ and $n\ge 2$, then
$\nim(\DNT(G))=\pty(m)$ and $\nim(\TER(G))=1+\pty(m)$.
\end{prop}

\begin{example}
The diamond graph shown in Figure \ref{fig:diamond} is the complete split graph $G:=K_2+\overline{K_2}$ with $\mathcal{G}=\{\{x,y\}\}$ and $\mathcal{N}^*=\{\{u,v\}\}$. We have $\nim(\DNT(G))=\pty(\{u,v\})=0$ and $\nim(\TER(G))=1+\pty(\{u,v\})=1$.
\end{example}

\begin{figure}[h]

\begin{tikzpicture}[scale=.8,auto]
\node (3) at (-1,.5) [vert ] {\scriptsize $x$};
\node (1) at (0,0) [vert] {\scriptsize $v$};
\node (2) at (0,1) [vert] {\scriptsize $u$};
\node (4) at (1,.5) [vert] {\scriptsize $y$};
\path [b] (4) to (1) to (2) to (3) to (1) to (4) to (2);
\end{tikzpicture}

\caption{
\label{fig:diamond}
The complete split graph $K_2+\overline{K_2}$ with a unique minimal generating set $L=\{x,y\}$.
}
\end{figure}
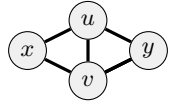

\subsection{Corona graphs}

The \emph{corona} $H \circ K_1$ is formed from the graph $H$ by adding for each $v \in V(H)$ a new vertex $v'$ and a new edge $vv'$. A corona graph has a unique minimal generating set $L=\{v'\mid v\in V(G)\}$ by~\cite[Proposition 6.10]{BEMSS2024}. Hence we have the following.

\begin{prop}
If $H$ is a nontrivial graph and $G=H\circ K_1$, then
$\nim(\DNT(G))=\pty(V)$ and $\nim(\TER(G))=1+\pty(V)$.
\end{prop}

\subsection{Block graphs}

A \emph{block} of a graph is a maximal connected subgraph without a cut vertex. A \emph{block graph} or \emph{clique tree} is a graph whose blocks are complete graphs. A vertex is called \emph{simplicial} if the subgraph induced by the neighbors of the vertex is a complete graph. The simplicial vertices of a block graph form the unique minimal generating set $L$ by ~\cite[Proposition 6.15]{BEMSS2024}.

\begin{example}\label{ex:complete graphs}
The complete graph $K_n$ is a block graph with $L=V$. Hence $\pty(L^c)=0$, and so $\nim(\DNT(K_n))=0$ and $\nim(\TER(K_n))=1$. 
\end{example}

\begin{example}
A generalized windmill graph $G = \Wd(\vec n)$ for  $\vec n = (n_1 , \ldots , n_{\ell} ) \in \mathbb{N}^{\ell}_{\ge 2}$ and $\ell \ge 2$ is the block graph built from the complete graphs $K_{n_1} , \ldots , K_{n_{\ell}}$ by gluing at a common vertex $c$. Since $L=\{c\}^c$, $\nim(\DNT(G))=1-\pty(V)$ and $ \nim(\TER(G))=2-\pty(V)$.
\end{example}

\begin{example}
A forest graph $G$ is a block graph where $L$ is the set of leaves. In particular, 
$\nim(\DNT(P_n))=\pty(n)$ and
$\nim(\TER(P_n))=1+\pty(n)$
for the path graph $P_n$, while
$\nim(\DNT(K_{1,n}))=1$ and
$\nim(\TER(K_{1,n}))=1$
for the star graph $K_{1,n}$.
\end{example}

\section{Graph families}\label{sec:GraphFamilies}

We study the impartial games on several graph families. 

\subsection{Cycle graphs}

For $n\ge 4$, we define $C_n$ to be the \emph{cycle graph} with vertex set $\{v_1,\ldots,v_n\}$ and $v_i$ is adjacent to $v_{i+1}$ if the indices are considered modulo $n$. Note that if $n=3$, then the construction gives the complete graph $K_3$.

\begin{prop}
For cycle graphs with odd $n$, 
\[
\mathcal{G}=\{\{v_{i},v_j,v_{k}\} \mid i < j < k \text{ and } j-i,k-j,i+n-k\le(n-1)/2 \}.
\]
\end{prop}

\begin{proof}
It is easy to see that two vertices cannot generate. The condition on the $i,j,k$ guarantees that the unique geodesic between two of the chosen vertices does not contain the third chosen vertex. Hence every such $\{v_i,v_j,v_k\}$ is a generating set. Every generating set must contain three such vertices.
\end{proof}

\begin{example}
For $C_5$, the maximum directed distance allowed between consecutive vertices in a minimal generating set is $(5-1)/2=2$. So 
\begin{align*}
\mathcal{G} &=\{\{v_1,v_3,v_4\},\{v_2,v_4,v_5\},\{v_1,v_3,v_5\},\{v_1,v_2,v_4\},\{v_2,v_3,v_5\}\}, \\
\mathcal{N}^* &=\complement(\mathcal{G})=\{\{v_2,v_5\},\{v_1,v_3\},\{v_2,v_4\},\{v_3,v_5\},\{v_1,v_4\}\}, \\
\mathcal{G}^* &=\Tr(\mathcal{G})=\{\{v_1,v_2\},\{v_2,v_3\},\{v_3,v_4\},\{v_4,v_5\},\{v_1,v_5\}\}.
\end{align*}
Note that $\mathcal{N}^*$ only contains even sets. Also note that the second player can win $\TER(C_5)$ after two moves since every vertex is contained in a two-element terminating set.
\end{example}

\begin{prop}
For cycle graphs, $\nim(\DNT(C_{n}))=0$.
\end{prop}

\begin{proof}
If $n$ is even, then the second player wins by always selecting the vertex antipodal to the vertex selected by the first player. The game ends when there are two antipodal unselected vertices remaining. 
If $n$ is odd, then $\mathcal{N}^*=\complement(\mathcal{G})$ contains sets with size $n-3$, so the result follows from Proposition~\ref{prop:all-same}. 
\end{proof}

\begin{prop}
For cycle graphs with even $n$, $\nim(\TER(C_{n}))=0$.
\end{prop}

\begin{proof}
The second player wins by selecting vertices antipodal to the selection of the first player until there are only four unselected vertices remaining. In the last move the second player selects a vertex that is not antipodal to the vertex selected by the first player.
\end{proof}

The case where $n$ is odd is surprisingly tricky for $\TER(C_n)$.   We have verified the following conjecture up to $n=21$ via computer.

\begin{conjecture}
\label{conj:oddCyclic}
For cycle graphs with odd $n$, $\nim(\TER(C_{n}))=0$.
\end{conjecture}

\subsection{Hypercube graphs}

For $n\geq 3$, we define the set of  binary strings of length $n$ via
\[
\{0,1\}^n := \{a_1a_2\cdots a_n \mid a_k \in \{0,1\} \}.
\]
The \emph{hypercube graph} $Q_n$ of dimension $n$ is the graph whose vertices are elements of $\{0,1\}^n$ with two binary strings connected by an edge exactly when they differ by a single digit. We say that two binary strings $a_1a_2\cdots a_n$ and $b_1b_2\cdots b_n$ are \emph{antipodal} if $a_i\neq b_i$ for all $1\leq i\leq n$.

\begin{prop}
For hypercube graphs, $\nim(\DNT(Q_{n}))=0$.
\end{prop}

\begin{proof}
Each pair of antipodal vertices is a minimal generating set by \cite[Proposition~7.7]{BEMSS2024}. Then the second player wins by always selecting the antipodal vertex after the first player's selection. The game ends when there is a single pair of antipodal vertices remaining.
\end{proof}
  
\begin{prop}
For hypercube graphs, $\nim(\TER(Q_{n}))=0$.
\end{prop}

\begin{proof}
The second player always selects the antipodal vertex after the first player's selection until there are two pairs of antipodal vertices remaining. After the first player chooses one of the remaining four vertices, the second player ends the game by selecting one of the two remaining vertices not antipodal to the first player's last choice.
\end{proof}


\subsection{Complete multipartite graphs}
In this section we consider the \emph{complete multipartite graph} $G=K_{m_1,\ldots,m_k}$ with $k\ge2$, $m_1\le m_2\le \cdots \le m_k$ and $m_k\ge 2$. The parts of $G$ that contain only one vertex are called small, while the parts containing at least two vertices are called large. We let $\sigma$ be the number of small parts, that is, $\sigma:=|\{i\mid m_i=1\}|$. We also let $\lambda$ be the number of large parts, that is, $\lambda:=|\{i\mid m_i\ge 2\}|$. Note that $\lambda\geq 1$ by definition of a complete multipartite graph.

If $\lambda=1$, then $G$ is a complete split graph $K_\sigma+\overline{K_{m_k}}$. So we have the following result by Proposition~\ref{prop:splitTER}.

\begin{prop}
For complete multipartite graphs with $\lambda=1$, $\nim(\DNT(G))=\pty(\sigma)$ and $\nim(\TER(G))=1+\pty(\sigma)$.
\end{prop}

\begin{prop}
\label{prop:TERMaximalsCompleteBipartiteManyBigComponents}
For complete multipartite graphs with $\lambda\ge 2$, $\mathcal{N}^\star$ consists of sets that are the complement of a set that contains exactly two elements from a single part.
\end{prop}

\begin{proof}
It is easy to see that $\mathcal{G}$ consists of the sets that contain exactly two elements from a single part of $G$. The result now follows from the equality $\mathcal{N}^\star=\complement(\mathcal{G})$.
\end{proof}

\begin{prop}
For complete multipartite graphs with $\lambda \geq 2$, $\nim(\DNT(G))=\pty(V)$. 
\end{prop}

\begin{proof}
This is a consequence of Propositions~\ref{prop:TERMaximalsCompleteBipartiteManyBigComponents} and~\ref{prop:all-same}.
\end{proof}

We are going to study $\TER(K_{m_1,\ldots,m_k})$ for $\lambda\ge 2$ through an option-preserving image.
Let $U$ be a finite multiset of nonnegative integers. In the \emph{multiset terminate game} $\TER(U)$, the players decrease one of the positive elements in $U$ by 1 in each turn. The game ends when all elements of $U$ are less than 2. 

For a position $P$ of $\TER(K_{m_1,\ldots,m_k})$, let $f(P)$ be the multiset consisting of the number of unmarked vertices in each component. Since $\mathcal{G}$ consists of the sets that contain exactly two elements from a single part of $G$,
\[
f:\TER(K_{m_1,\ldots,m_k})\to\TER(\m{m_1,\ldots,m_k})
\]
is an option-preserving map if $\lambda\ge 2$.

\begin{example}
Figure~\ref{fig:fimage1} shows a play in $\TER(K_{1,2,2})$ and its image in $\TER(\m{1,2,2})$ under the option-preserving map $f$. Note that $\sigma=1$ and $\lambda=2$.
\end{example}

\begin{figure}[h]
\begin{tabular}{ccccccc}
\begin{tikzpicture}[xscale=.9,auto,,baseline=.7cm]
\node (1) at (1,1.5) [small vert] {};
\node (2) at (0,.8) [small vert] {};
\node (3) at (0.3,0.3) [small vert] {};
\node (4) at (2,.8) [small vert] {};
\node (5) at (1.7,0.3) [small vert] {};
\path [b] (1) to (2);
\path [b] (1) to (3);
\path [b] (1) to (4);
\path [b] (1) to (5);
\path [b] (2) to (4);
\path [b] (2) to (5);
\path [b] (3) to (4);
\path [b] (3) to (5);
\end{tikzpicture}
& $\longrightarrow$ &
\begin{tikzpicture}[xscale=.9,auto,,baseline=.7cm]
\node (1) at (1,1.5) [small vert,fill=purple] {};
\node (2) at (0,.8) [small vert] {};
\node (3) at (0.3,0.3) [small vert] {};
\node (4) at (2,.8) [small vert] {};
\node (5) at (1.7,0.3) [small vert] {};
\path [b] (1) to (2);
\path [b] (1) to (3);
\path [b] (1) to (4);
\path [b] (1) to (5);
\path [b] (2) to (4);
\path [b] (2) to (5);
\path [b] (3) to (4);
\path [b] (3) to (5);
\end{tikzpicture}
& $\longrightarrow$ &
\begin{tikzpicture}[xscale=.9,auto,,baseline=.7cm]
\node (1) at (1,1.5) [small vert,fill=purple] {};
\node (2) at (0,.8) [small vert,fill=purple] {};
\node (3) at (0.3,0.3) [small vert] {};
\node (4) at (2,.8) [small vert] {};
\node (5) at (1.7,0.3) [small vert] {};
\path [b] (1) to (2);
\path [b] (1) to (3);
\path [b] (1) to (4);
\path [b] (1) to (5);
\path [b] (2) to (4);
\path [b] (2) to (5);
\path [b] (3) to (4);
\path [b] (3) to (5);
\end{tikzpicture}
& $\longrightarrow$ &
\begin{tikzpicture}[xscale=.9,auto,,baseline=.7cm]
\node (1) at (1,1.5) [small vert,fill=purple] {};
\node (2) at (0,.8) [small vert,fill=purple] {};
\node (3) at (0.3,0.3) [small vert] {};
\node (4) at (2,.8) [small vert] {};
\node (5) at (1.7,0.3) [small vert,fill=purple] {};
\path [b] (1) to (2);
\path [b] (1) to (3);
\path [b] (1) to (4);
\path [b] (1) to (5);
\path [b] (2) to (4);
\path [b] (2) to (5);
\path [b] (3) to (4);
\path [b] (3) to (5);
\end{tikzpicture}
\\
\vspace{-3mm}
&&&&&& \\
$\m{1,2,2}$ & $\longrightarrow$ & $\m{0,2,2}$ & $\longrightarrow$ & $\m{0,1,2}$ & $\longrightarrow$ & $\m{0,1,1}$ \end{tabular}

\caption{
\label{fig:fimage1}
A play in $\TER(K_{1,2,2})$ and its image in $\TER(\m{1,2,2})$.
}
\end{figure}
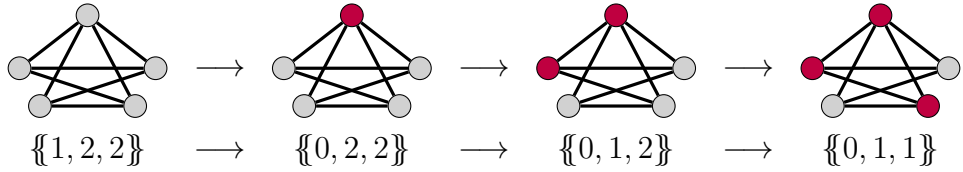

The maximum element of the multiset $U$ is denoted by $m$. We define 
\[
\sigma:=|\m{u\in U\mid u=1}|,\quad \lambda:=|\m{u\in U\mid u\ge 2}|, \quad \nu:=|\m{u\in U\mid u=m}|-1, 
\]
and $p:=\pty(\|U\|)$, where $\|U\|:=\sum U$.

\begin{example}
If
\[
U=\m{0,\underbrace{1,1}_\sigma, \underbrace{2,2,2, \overbrace{3,3}^\nu ,3}_\lambda},
\]
then $\sigma=2$, $\lambda=6$, $\nu=2$, $m=3$, and $p=1$.
\end{example}

If $U$ is not terminal, then we define
\[
d:=m-\sum_{u\in U'} (u-1),
\text{ with } 
U':=\m{u\in U \mid u\ge 2}\setminus \m{m}. 
\]
If $U$ is terminal, then we let $d:=\infty$. The \emph{signature} of $U$ is $\rho(U):=(p,d)$. 

\begin{example}
If $U=\m{1,1,2,3,3}$, then $\rho(U)=(0,0)$ because $p=\pty(10)=0$, $U'=\{2,3\}$, and $d=3-(3-1)-(2-1)$. 

If $U=\m{0,1,1,1}$, then $\rho(U)=(1,\infty)$ because $p=\pty(3)=1$ and $U$ is a terminal position.
\end{example}

\begin{prop}
If $U$ is not a terminal position in the multiset terminate game and $(p,d)=\rho(U)$, then 
\[
\nim(U)=\begin{cases}
0, & p=0 \text{ and } d \leq 1 \\
1, & p=0 \text{ and } 2 \leq d \\
1, & p=1 \text{ and } d \leq 0 \\
2, & p=1 \text{ and } d \in \{1, 2\} \\
0, & p=1 \text{ and } 3 \leq d.
\end{cases}
\]
\end{prop}

\begin{table}
\footnotesize
\renewcommand{\arraystretch}{1.3}
\setlength{\tabcolsep}{3pt}
\begin{tabular}{|c|c||c||c|c||c|c||c|c||c|c||c|c|}
\hline 
\multicolumn{2}{|c||}{$\rho(U)$} &  & \multicolumn{2}{c||}{$\sigma>0$} & \multicolumn{2}{c||}{$\lambda>1$} & \multicolumn{2}{c||}{$\nu=0$} & \multicolumn{2}{c||}{%
\renewcommand{\arraystretch}{.7}
\begin{tabular}{c}
$\lambda=1$\\
$m=2$\\
\end{tabular}} & \multicolumn{2}{c|}{$\nim(U)$}\\
\hline 
$p$ & $d$ & $\tilde{p}$ & $\tilde{d}$ & $\nim(\tilde{U})$ & $\tilde{d}$ & $\nim(\tilde{U})$ & $\tilde{d}$ & $\nim(\tilde{U})$ & $\tilde{d}$ & $\nim(\tilde{U})$ & $\nim\circ\Opt$ & mex\\
 &  & $1-p$ & $d$ &  & $d+1$ &  & $d-1$ &  & $\infty$ &  &  & \\
\hline 
0 & $\le-1$ & 1 & $\le-1$ & 1 & $\le0$ & 1 & $\le-2$ & 1 &  &  & $\emptyset,\{1\}$ & 0\\
0 & 0 & 1 & 0 & 1 & 1 & 2 & $-1$ & 1 &  &  & $\emptyset,\{1,2\}$ & 0\\
0 & 1 & 1 & 1 & 2 & 2 & 2 & 0 & 1 &  &  & $\emptyset,\{1,2\}$ & 0\\
\hline 
0 & 2 & 1 & 2 & 2 & 3 & $\text{\fbox{0}}$ & 1 & 2 & $\infty$ & $\text{\fbox{0}}$ & $\{0\}$,$\{0,2\}$ & 1 \\
0 & 3 & 1 & 3 & $\text{\fbox{0}}$ & 4 & $\text{\fbox{0}}$ & 2 & 2 &  &  & $\{0\}$,$\{0,2\}$ & 1 \\
0 & $\ge4$ & 1 & $\ge4$ & $\text{\fbox{0}}$ & $\ge5$ & $\text{\fbox{0}}$ & $\ge3$ & $\text{\fbox{0}}$ &  &  & $\{0\}$ & 1 \\
\hline 
1 & $\le0$ & 0 & $\le0$ & $\text{\fbox{0}}$ & $\le1$ & $\text{\fbox{0}}$ & $\le-1$ & $\text{\fbox{0}}$ &  &  & $\{0\}$ & 1\\
\hline 
1 & 1 & 0 & 1 & $\text{\fbox{0}}$ & 2 & $\text{\dbox{1}}$ & 0 & $\text{\fbox{0}}$ &  &  & $\{0,1\}$ & 2 \\
1 & 2 & 0 & 2 & 1 & 3 & $\text{\fbox{1}}$ & 1 & $\text{\dbox{0}}$ & $\infty$ & 0 & $\{0,1\}$ & 2\\
\hline 
1 & $\ge3$ & 0 & $\ge3$ & $\text{\fbox{1}}$ & $\ge4$ & $\text{\fbox{1}}$ & $\ge2$ & $\text{\fbox{1}}$ &  &  & $\{1\}$ & 0\\
\hline 
\end{tabular}

\bigskip
\caption{
\label{tab:MultiPartCases}
Case analysis for the computation of $\nim(U)$ based on the signatures of $U$ and $\tilde{U}$. The column labeled with $\nim\circ\Opt$ shows a subset and a superset of $\nim(\Opt(U))$.
}
\end{table}

\begin{proof}
We argue by induction on $\|U\|$. Since $U$ is not terminal, $\lambda\ge 1$ and $m\ge 2$. For a possible option $\tilde{U}$ of $U$, we let $(\tilde{p},\tilde{d}):=\rho(\tilde{U})$. Note that $\|\tilde{U}\|=\|U\|-1$ and $\tilde{p}=1-p$. Let $D:=\{\tilde{d}\mid \tilde{U}\in \Opt(U)\}$. It is easy to check that
\begin{enumerate}
\item
$d-1\in D$ if and only if $\nu=0$;
\item
$d\in D$ if and only if $\sigma>0$;
\item
$d+1\in D$ if and only if $\lambda>1$;
\item
$\infty\in D$ if and only if $\lambda=1$ and $m=2$. 
\end{enumerate}
Table~\ref{tab:MultiPartCases} shows the computation of $\nim(U)=\mex(\nim(\Opt(U)))$ depending on the signature of $U$ using induction. Each $\tilde{d}$ column shows a possibility that might or might not occur depending on the condition shown at the top of the column. The corresponding nim-values  provide a superset of $\nim(\Opt(U))$. Nim-values contained in the same type of box indicate that one of the possible $\tilde{d}$ values must occur. The corresponding nim-values provide a subset of $\nim(\Opt(U))$. We justify the three nontrivial such claims below.

First, consider the $\rho(U)=(0,2)$ case. Suppose $3\not\in D$ and $\infty\not\in D$. Then $\lambda=1$ and $m>2$. This gives the contradiction $2=d=m>2$.

Next, consider the $\rho(U)=(0,3)$ case. Suppose $3\not\in D$ and $4\not\in D$. Then $\sigma=0$, $\lambda=1$ and $m=3$. This gives the contradiction $p=1$.

Finally, consider the $\rho(U)=(1,1)$ case. Since $\lambda=1$ implies $d=m>2$, we must have $\lambda>1$ and so $2\in D$. 
Suppose $1\not\in D$ and $0\not\in D$. Then $\sigma=0$ and $\nu\ge 1$. Since $p=1$, we must have $\lambda\ge 3$. This gives the contradiction $1=d<m-(m-1)-1=0$.
\end{proof}

The next result is an immediate consequence.  

\begin{prop}
For complete multipartite graphs with $\lambda \geq 2$,
\[
\nim(\TER(K_{m_1,\ldots,m_k}))=\begin{cases}
0, & |V| \text{ even and } d \leq 1 \\
1, & |V| \text{ even and } 2 \leq d \\
1, & |V| \text{ odd and } d \leq 0 \\
2, & |V| \text{ odd and } d \in \{1, 2\} \\
0, & |V| \text{ odd and } 3 \leq d,
\end{cases}
\]
where $d:=m_k - \sum_{i=1}^{k-1}(m_i-1)$.
\end{prop}

\subsection{Wheel graphs}
For $n \geq 5$, we define $W_n$ to be the \emph{wheel graph} with $n$ total vertices $\{v_1,\ldots,v_{n-1},c\}$, where $c$ is the center and $v_i$ is adjacent to $v_{i+1}$  if the indices are considered modulo $n-1$. 

\begin{prop}
\cite[Proposition 7.22]{BEMSS2024}
For wheel graphs, $\mathcal{N}$ consists of the complements of sets containing two neighboring non-central vertices.  
\end{prop}

The following is an immediate consequence since $\mathcal{G}^*=\complement(\mathcal{N})$. 

\begin{corollary}
For wheel graphs, $\mathcal{G}^*$ consists of pairs of neighboring non-central vertices.
\end{corollary}

We will prove the following result for all $n$ in Proposition~\ref{thm:wheel-dnt-general}. We include the alternate proof because of its simplicity.

\begin{prop}
\label{wheel-dnt-odd}                   
For wheel graphs with odd $n$, $\nim(\DNT(W_{n}))=1$.
\end{prop}

\begin{proof}
Since $n$ is odd, $W_n$ consists of an even cycle with a center vertex. The second player wins $\DNT(W_n)+*1$ using a pairing strategy, where a vertex on the rim is paired with the antipodal vertex and the  center vertex is paired with the stone from $*1$. The second player always selects the pair of the element chosen by the first player. The symmetry of the pairing of the rim vertices guarantees that the move of the second player prescribed by the strategy is always allowed.
\end{proof}

Berlekamp, Conway, and Guy~\cite{ww1} define the game Dawson's Chess $\dawsonschess(n)$ as follows. Two players start with $n \geq 0$ pins arranged in a row. They alternate, and on each turn they  knock down a pin and any of its standing adjacent neighbors.  The player to knock down the last pin wins. Our interest in Dawson's Chess comes from the next result.

If $P$ is a position of a removing game, then we use the notation $\DNT_P(G)$ and $\TER_P(G)$ to denote the games that start at position $P$.  

\begin{prop}
\label{prop:dawsonchessplusoneequivalence}
There is a surjective option-preserving map $f:\DNT_{\{v_{k}\}}(W_n)\to\dawsonschess(n-4)+*1$ for all $1 \leq k \leq n-1$.
\end{prop}

\begin{proof}
For notational convenience, assume without loss of generality that $k=n-2$. The vertices $v_{n-3}$ and $v_{n-1}$ cannot be selected since they are adjacent to $v_{n-2}$ and would result in termination.  Thus, the only legal moves in $\DNT_{\{v_{n-2}\}}(W_n)$ are in $\{v_1,\ldots, v_{n-4}, c\}$.  Let $p_1,\ldots,p_{n-4}$ be the pins of $\dawsonschess(n-4)$. For $P\in\DNT_{\{v_{n-2}\}}(W_n)$, let $f(P)$ be $\dawsonschess(n-4)+*1$ defined as follows. Pin $p_i$ of $f(P)$ is knocked down if and only if vertex $v_j$ of $W_n$ is selected in $P$ for some $|i-j|\le 1$. The single stone from $*1$ is taken if and only if vertex $c$ of $W_n$ is selected in $P$. 

Note that $f$ is an option-preserving map, since $v_i \in P$ implies that neither $v_{i-1}$ nor $v_{i+1}$ can be selected.  Since $v_i \in P$, this means that $p_{i-1}$, $p_i$, and $p_{i+1}$ have been knocked down in $f(P)$.   Since a vertex $v_j$ may be selected if and only if neither $v_{j-1}$ nor $v_{j+1}$ have been selected, a pin $p_j$ may be selected if and only if $p_{j-1}$ and $p_{j+1}$ have been selected. 
\end{proof}

The proof of the following is similar to that of the previous result.

\begin{prop}
There is a surjective option-preserving map $f:\DNT_{\{c,v_{k}\}}(W_n)\to\dawsonschess(n-4)$ for all $1 \leq k \leq n-1$.
\end{prop}

\begin{example}
Figure~\ref{fig:map} shows the surjective option-preserving map 
$$
f:\DNT_{\{v_1\}}(W_7)\to\dawsonschess(3)+*1.
$$
The restriction $\DNT_{\{c,v_1\}}(W_7)\to\dawsonschess(3)$ is also surjective and option preserving. 
\end{example}

\begin{figure}[h]
\begin{tikzpicture}[xscale=2,yscale=1.5]
\draw[rounded corners,color=yellow,fill=yellow!20] (-.5,-.3) -- (.5,-.3) -- (.5,.3) -- (-.5,.3) -- cycle; 
\draw[rounded corners,color=yellow,fill=yellow!20] (-.5,-1.3) -- (.5,-1.3) -- (.5,-.68) -- (-.5,-.68) -- cycle; 

\node[inner sep=-1.5,circle] (t) at (0,2) {
\begin{tikzpicture}
\node at (0:0) {$\circ$};
\node at (30:.3) {$\circ$};
\node at (90:.3) {$\bullet$};
\node at (150:.3) {$\circ$};
\node at (210:.3) {$\circ$};
\node at (270:.3) {$\circ$};
\node at (330:.3) {$\circ$};
\end{tikzpicture}
};
\node[inner sep=-1.5,circle] (m1) at (-1,1) {
\begin{tikzpicture}
\node at (0:0) {$\circ$};
\node at (30:.3) {$\circ$};
\node at (90:.3) {$\bullet$};
\node at (150:.3) {$\circ$};
\node at (210:.3) {$\bullet$};
\node at (270:.3) {$\circ$};
\node at (330:.3) {$\circ$};
\end{tikzpicture}
};
\node[inner sep=-1.5,circle] (m2) at (-.3,0) {
\begin{tikzpicture}
\node at (0:0) {$\circ$};
\node at (30:.3) {$\circ$};
\node at (90:.3) {$\bullet$};
\node at (150:.3) {$\circ$};
\node at (210:.3) {$\bullet$};
\node at (270:.3) {$\circ$};
\node at (330:.3) {$\bullet$};
\end{tikzpicture}
};
\node[inner sep=-1.5,circle] (m2+) at (.3,0) {
\begin{tikzpicture}
\node at (0:0) {$\circ$};
\node at (30:.3) {$\circ$};
\node at (90:.3) {$\bullet$};
\node at (150:.3) {$\circ$};
\node at (210:.3) {$\circ$};
\node at (270:.3) {$\bullet$};
\node at (330:.3) {$\circ$};
\end{tikzpicture}
};
\node[inner sep=-1.5,circle] (m3) at (1,1) {
\begin{tikzpicture}
\node at (0:0) {$\circ$};
\node at (30:.3) {$\circ$};
\node at (90:.3) {$\bullet$};
\node at (150:.3) {$\circ$};
\node at (210:.3) {$\circ$};
\node at (270:.3) {$\circ$};
\node at (330:.3) {$\bullet$};
\end{tikzpicture}
};
\node[inner sep=-1.4,circle] (m4) at (0,1) {
\begin{tikzpicture}
\node at (0:0) {$\bullet$};
\node at (30:.3) {$\circ$};
\node at (90:.3) {$\bullet$};
\node at (150:.3) {$\circ$};
\node at (210:.3) {$\circ$};
\node at (270:.3) {$\circ$};
\node at (330:.3) {$\circ$};
\end{tikzpicture}
};
\node[inner sep=-1.5,circle] (b1) at (-1,0) {
\begin{tikzpicture}
\node at (0:0) {$\bullet$};
\node at (30:.3) {$\circ$};
\node at (90:.3) {$\bullet$};
\node at (150:.3) {$\circ$};
\node at (210:.3) {$\bullet$};
\node at (270:.3) {$\circ$};
\node at (330:.3) {$\circ$};
\end{tikzpicture}
};
\node[inner sep=-1.5,circle] (b2) at (1,0) {
\begin{tikzpicture}
\node at (0:0) {$\bullet$};
\node at (30:.3) {$\circ$};
\node at (90:.3) {$\bullet$};
\node at (150:.3) {$\circ$};
\node at (210:.3) {$\circ$};
\node at (270:.3) {$\circ$};
\node at (330:.3) {$\bullet$};
\end{tikzpicture}
};
\node[inner sep=-1.5,circle] (b3) at (.3,-1) {
\begin{tikzpicture}
\node at (0:0) {$\bullet$};
\node at (30:.3) {$\circ$};
\node at (90:.3) {$\bullet$};
\node at (150:.3) {$\circ$};
\node at (210:.3) {$\bullet$};
\node at (270:.3) {$\circ$};
\node at (330:.3) {$\bullet$};
\end{tikzpicture}
};
\node[inner sep=-1.5,circle] (b3+) at (-.3,-1) {
\begin{tikzpicture}
\node at (0:0) {$\bullet$};
\node at (30:.3) {$\circ$};
\node at (90:.3) {$\bullet$};
\node at (150:.3) {$\circ$};
\node at (210:.3) {$\circ$};
\node at (270:.3) {$\bullet$};
\node at (330:.3) {$\circ$};
\end{tikzpicture}
};
\draw[->] (t) -- (m1);
\draw[->] (t) to[bend left=18] (m2+);
\draw[->] (t) -- (m3);
\draw[->] (t) -- (m4);
\draw[->] (m1) -- (b1);
\draw[->] (m4) -- (b1);
\draw[->] (m3) -- (b2);
\draw[->] (m4) -- (b2);
\draw[->] (m2) -- (b3);
\draw[->] (m4) to[bend left=12] (b3+);
\draw[->] (m1) -- (m2);
\draw[->] (m3) -- (m2);
\draw[->] (b1) -- (b3);
\draw[->] (b2) -- (b3);
\draw[->] (m2+) -- (b3+);
\end{tikzpicture}
\hfil  
\newcommand\PP{\scalebox{.7}{\BlackPawnOnWhite}}%
\newcommand\pp{\scalebox{.7}{\color{grey}\BlackPawnOnWhite}}%
\begin{tikzpicture}[xscale=2,yscale=1.5]
\node[draw,rounded corners,inner sep=1.5] (t) at (0,2) {$\underset{+*1}{\sympawn\sympawn\sympawn}$};
\node[draw,rounded corners,inner sep=1.5] (m1) at (-1,1) {$\underset{+*1}{{\color{grey!70}{\sympawn\sympawn}}\sympawn}$};
\node[draw,rounded corners,inner sep=1.5]  (m2) at (0,0) {$\underset{+*1}{{\color{grey!70}{\sympawn\sympawn\sympawn}}}$};
\node[draw,rounded corners,inner sep=1.5]  (m3) at (1,1) {$\underset{+*1}{\sympawn{\color{grey!70}{\sympawn\sympawn}}}$};
\node[draw,rounded corners,inner sep=1.5]  (m4) at (0,1) {$\underset{+*0}{\sympawn\sympawn\sympawn}$};
\node[draw,rounded corners,inner sep=1.5]  (b1) at (-1,0) {$\underset{+*0}{{\color{grey!70}{\sympawn\sympawn}}\sympawn}$};
\node[draw,rounded corners,inner sep=1.5]  (b2) at (1,0) {$\underset{+*0}{\sympawn{\color{grey!70}{\sympawn\sympawn}}}$};
\node[draw,rounded corners,inner sep=1.5]  (b3) at (0,-1) {$\underset{+*0}{{\color{grey!70}{\sympawn\sympawn\sympawn}}}$};
\draw[->] (t) -- (m1);
\draw[->] (t) to[bend left=35] (m2);
\draw[->] (t) -- (m3);
\draw[->] (t) -- (m4);
\draw[->] (m1) -- (b1);
\draw[->] (m4) -- (b1);
\draw[->] (m3) -- (b2);
\draw[->] (m4) -- (b2);
\draw[->] (m2) -- (b3);
\draw[->] (m4) to[bend right=35] (b3);
\draw[->] (m1) -- (m2);
\draw[->] (m3) -- (m2);
\draw[->] (b1) -- (b3);
\draw[->] (b2) -- (b3);
\end{tikzpicture}

\caption{
\label{fig:map}
Gamegraphs $\DNT_{\{v_0\}}(W_7)$ and $\dawsonschess(3)+*1$. The positions boxed together in the first gamegraph are mapped to the same position on the second gamegraph by the option-preserving map $f$. The grayed out pins indicate that they have been knocked down.
}
\end{figure}
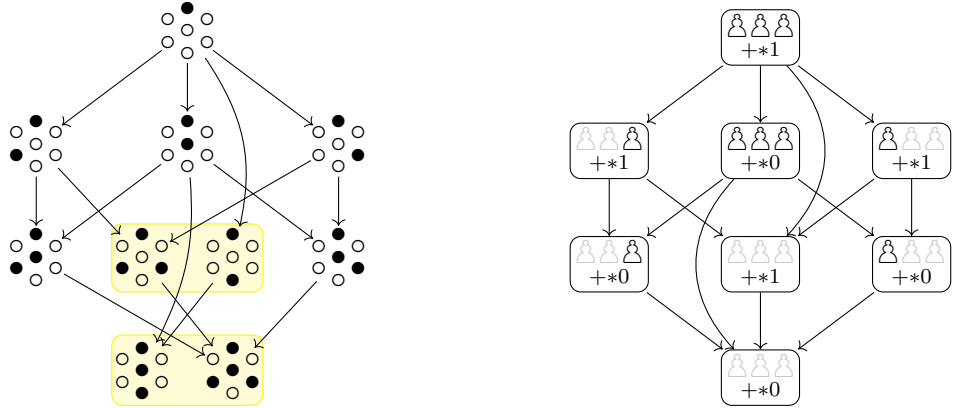

\begin{prop}
\label{prop:dawsonschessequivalence}
For wheel graphs,  
\[
\nim(\DNT(W_n))=\begin{cases}
0, & \nim(\dawsonschess(n-4))=0\\
1, & \text{otherwise}.
\end{cases}
\]
\end{prop}

\begin{figure}[h]
\begin{tikzpicture}[xscale=1.8,yscale=1.3]
\node (c) at (0,-.3) {$\underset{\color{cyan}*0 \text{ or } *1}{\DNT(W_n)}$};
\node (l) at (2,-1) {$\underset{\color{cyan}*1 \text{ or } *(\ne 1)}{\dawsonschess(n-4)+ *1}$};
\node (r) at (-1.5,-1) {$\underset{\color{cyan}*1 \text{ or } *0}{\bullet}$};
\node (rr) at (-1.5,-2) {$\underset{\color{cyan}*0 \text{ or } *(\ne 0)}{\dawsonschess(n-4)}$};
\draw[->] (c) -- node[above right] {$\scriptstyle v_{n-1}$} (l);
\draw[->] (c) -- node[above left] {$\scriptstyle c$} (r);
\draw[->] (r)  -- (rr);
\end{tikzpicture}

\caption{
\label{fig:WheelDNTCases}
Schematic gamegraph for $\DNT(W_n)$, where each position is indicated by an isomorphic game. 
}
\end{figure}
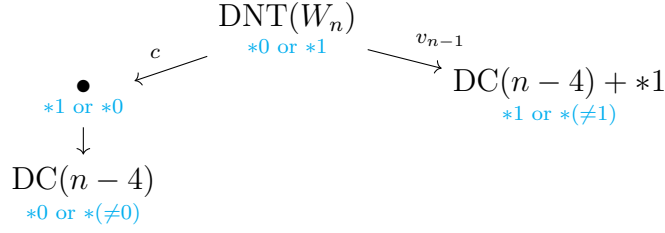

\begin{proof}
Without loss of generality, the first player will select either $c$ or $v_{n-1}$ on the first move. This creates position $P=\{c\}$ or $Q=\{v_{n-1}\}$, respectively, as shown in Figure~\ref{fig:WheelDNTCases}. 
In position $P$ the second player will select a rim vertex, which creates a position equivalent to $\dawsonschess(n-4)$. Position $Q$ is equivalent to $\dawsonschess(n-4)+*1$ by Proposition~\ref{prop:dawsonchessplusoneequivalence}. 

If $\nim(\dawsonschess(n-4))=0$, then $\nim(P)=1=\nim(Q)$ so that $\nim(\DNT(W_n))=\mex\{1\}=0$.
If $\nim(\dawsonschess(n-4))>0$, then $\nim(P)=0$ and $\nim(Q)\ne 1$ so that $\nim(\DNT(W_n))=\mex\{0,\nim(Q)\}=1$.
\end{proof}

\begin{prop}\label{prop:dawsonschess-sequence}\cite[Page~89]{ww1}
The sequence of nim-values of $\dawsonschess(n)$ is eventually periodic with preperiod 
\begin{gather*}
\hat{0}1120311033224\hat{0}5\hat{2}\hat{2}3301130211045\hat{2}74
\\
\hat{0}1120311033224\hat{4}5\hat{5}\hat{2}3301130211045\hat{3}74
\end{gather*}
of length 68 and period 
\[
\hat{8}1120311033224\hat{4}5\hat{5}\hat{9}3301130211045\hat{3}74
\]
of length 34.
\end{prop}

The hats in the previous result indicate where the entries differ.
Note that the following theorem is consistent with Proposition~\ref{wheel-dnt-odd}.  

\begin{thm}\label{thm:wheel-dnt-general}
For wheel graphs,  
\[
\nim(\DNT(W_n))=\begin{cases}
0, & n \in \{18,38\} \text{ or } n \text{ mod }  34 \in  \{8, 12,24,28,32\}\\
1, & \text{otherwise}.
\end{cases}
\]
\end{thm}                   

\begin{proof}
By Proposition~\ref{prop:dawsonschessequivalence}, $\nim(\DNT(W_n))=0$ exactly when $\nim(\dawsonschess(n-4))=0$ and equals $1$ otherwise.  By Proposition~\ref{prop:dawsonschess-sequence}, $\nim(\dawsonschess(n-4))=0$ exactly in the cases in the statement of the theorem.
\end{proof}

The \emph{ordinal sum} $\mathsf{G}:\mathsf{H}$ of the gamegraphs $\mathsf{G}$ and $\mathsf{H}$ \cite{OrdinalSumImpartial,ONAG,ordinalSum} has position set $\mathsf{G}\times\mathsf{H}$ and 
\[
\Opt(p,q):=\begin{cases}
\Opt(p)\times\{q\}, & p\ne p_0 \\
(\Opt(p_0)\times\{q\})\cup(\{p_0\}\times\Opt(q)), & p=p_0,
\end{cases}
\]
where $p_0$ is the starting position of $\mathsf{G}$. This means a player can make a move either in $\mathsf{G}$ or in $\mathsf{H}$ but $\mathsf{H}$ is discarded as soon as a move is made in $\mathsf{G}$.

\begin{example}
Figure~\ref{fig:ordinalProd} shows the gamegraph $\mathsf{G}$ and the ordinal sum $*2:\mathsf{G}$. Positions of the form $(p_0,q)=(*2,q)$ in $*2:\mathsf{G}$ are shown as squares. The arrows between these squares represent moves in $\mathsf{G}$. The other arrows represent moves in $*2$. 
\end{example}

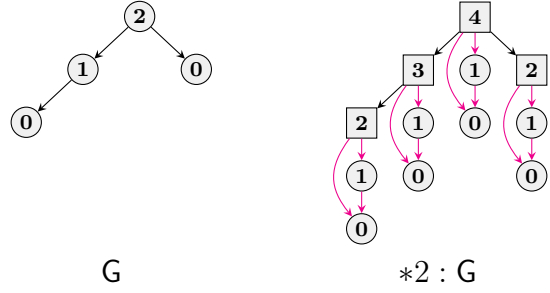
\begin{figure}[h]
\begin{tabular}{ccc}
\begin{tikzpicture}[xscale=1.5, yscale=.7,baseline=-46]
\node[med vert] (2)  at (1,2) {$\scriptstyle \mathbf{2}$};
\node[med vert] (a)  at (1.5,1) {$\scriptstyle \mathbf{0}$};
\node[med vert] (1)  at (0.5,1) {$\scriptstyle \mathbf{1}$};
\node[med vert] (0)  at (0,0) {$\scriptstyle \mathbf{0}$};
\path[a] (2) to (a);
\path[a] (2) to (1);
\path[a] (1) to (0);
\end{tikzpicture}
& $\qquad$ & 
\begin{tikzpicture}[xscale=1.5, yscale=.7]
\node (02) [med vert,rectangle]  at (1,2) {$\scriptstyle \mathbf{4}$};
\node (0a) [med vert,rectangle]  at (1.5,1) {$\scriptstyle \mathbf{2}$};
\node (01) [med vert,rectangle] at (0.5,1) {$\scriptstyle \mathbf{3}$};
\node (00) [med vert,rectangle] at (0,0) {$\scriptstyle \mathbf{2}$};

\node (12) [med vert]  at (1,1) {$\scriptstyle \mathbf{1}$};
\node (1a) [med vert]  at (1.5,0) {$\scriptstyle \mathbf{1}$};
\node (11) [med vert] at (.5,0) {$\scriptstyle \mathbf{1}$};
\node (10) [med vert] at (0,-1) {$\scriptstyle \mathbf{1}$};
 
\node (22) [med vert]  at (1,0) {$\scriptstyle \mathbf{0}$};
\node (2a) [med vert]  at (1.5,-1) {$\scriptstyle \mathbf{0}$};
\node (21) [med vert] at (.5,-1) {$\scriptstyle \mathbf{0}$};
\node (20) [med vert] at (0,-2) {$\scriptstyle \mathbf{0}$};
 
\path[a,magenta] (02) to (12);
\path[a,magenta] (01) to (11);
\path[a,magenta] (0a) to (1a);
\path[a,magenta] (00) to (10);
\path[a,magenta] (12) to (22);
\path[a,magenta] (11) to (21);
\path[a,magenta] (1a) to (2a);
\path[a,magenta] (10) to (20);
\path[a,magenta] (02) to[bend right=20] (22); 
\path[a,magenta] (01) to[bend right=20] (21); 
\path[a,magenta] (0a) to[bend right=20] (2a); 
\path[a,magenta] (00) to[bend right=20] (20); 
  
\path[a] (02) to (01);
\path[a] (02) to (0a);
\path[a] (01) to (00);
\end{tikzpicture}
\\
$\mathsf{G}$ & & $*2:\mathsf{G}$  
\end{tabular}

\caption{
\label{fig:ordinalProd}
The nim-values of the positions of $\mathsf{G}$ and the ordinal sum $*2:\mathsf{G}$. 
}
\end{figure}

This example suggests the following.

\begin{prop}
\label{prop:AddWinningOption}
If $\mathsf{G}$ is a gamegraph, then $\nim(*k:\mathsf{G})=k+\nim(\mathsf{G})$.
\end{prop}

\begin{proof}
It is easy to see that $\nim(*j,p)=j$ for all $j\in\{0,\ldots,k-1\}$ and $p\in\mathsf{G}$. Structural induction shows that
\[
\begin{aligned}
\nim(*k,p) 
& =\mex(\nim(\Opt(*k)\times\{p\})\cup\nim(\{*k\}\times\Opt(p))) \\
& =\mex(\{\nim(*j,p)\mid *j\in\Opt(*k)\} \cup\{\nim(*k,q)\mid q\in\Opt(p) \}) \\
& =\mex(\{0,\ldots,k-1\}\cup\{k+\nim(q)\mid q\in\Opt(p)\}) \\
& = k+\nim(p). 
\end{aligned}
\]
Hence $\nim(*k:\mathsf{G})=\nim(*k,p_0)=k+\nim(p_0)=k+\nim(\mathsf{G})$, where $p_0$ is the starting position of $\mathsf{G}$.
\end{proof}

\begin{prop}
\label{prop:sutjOP}
The games $\TER_{\{v_i\}}(W_n)$ and $*1:\DNT_{\{v_i\}}(W_n)$ have the same nim-value.
\end{prop}

\begin{proof}
Identifying the terminal positions in each game creates two isomorphic quotient gamegraphs. The canonical quotient maps are option-preserving and hence nim-value preserving.
\end{proof}

\begin{example}
Figure~\ref{fig:mapOrd} shows the gamegraphs for $\TER_{\{v_i\}}(W_5)$ and $*1:\DNT_{\{v_i\}}(W_5)$. Vertex $v_i$ is drawn on top. The two quotient maps identify all the shaded terminal positions.
\end{example}

\begin{figure}[h]
\begin{tikzpicture}[xscale=1.4,yscale=1.2]
\node[inner sep=-1.5,circle] (t) at (0,2) {
\begin{tikzpicture}
\node at (0:0) {$\circ$};
\node at (0:.3) {$\circ$};
\node at (90:.3) {$\bullet$};
\node at (180:.3) {$\circ$};
\node at (270:.3) {$\circ$};
\end{tikzpicture}
};
\node[draw=yellow,fill=yellow!20,inner sep=-1.5,circle] (m1) at (-.5,1) {
\begin{tikzpicture}
\node at (0:0) {$\circ$};
\node at (0:.3) {$\circ$};
\node at (90:.3) {$\bullet$};
\node at (180:.3) {$\bullet$};
\node at (270:.3) {$\circ$};
\end{tikzpicture}
};
\node[inner sep=-1.5,circle] (m2) at (-1.5,1) {
\begin{tikzpicture}
\node at (0:0) {$\circ$};
\node at (0:.3) {$\circ$};
\node at (90:.3) {$\bullet$};
\node at (180:.3) {$\circ$};
\node at (270:.3) {$\bullet$};
\end{tikzpicture}
};
\node[inner sep=-1.5,circle] (m3) at (1.5,1) {
\begin{tikzpicture}
\node at (0:0) {$\bullet$};
\node at (0:.3) {$\circ$};
\node at (90:.3) {$\bullet$};
\node at (180:.3) {$\circ$};
\node at (270:.3) {$\circ$};
\end{tikzpicture}
};
\node[draw=yellow,fill=yellow!20,inner sep=-1.5,circle] (m4) at (.5,1) {
\begin{tikzpicture}
\node at (0:0) {$\circ$};
\node at (0:.3) {$\bullet$};
\node at (90:.3) {$\bullet$};
\node at (180:.3) {$\circ$};
\node at (270:.3) {$\circ$};
\end{tikzpicture}
};
\node[draw=yellow,fill=yellow!20,inner sep=-1.4,circle] (b1) at (-2,0) {
\begin{tikzpicture}
\node at (0:0) {$\circ$};
\node at (0:.3) {$\circ$};
\node at (90:.3) {$\bullet$};
\node at (180:.3) {$\bullet$};
\node at (270:.3) {$\bullet$};
\end{tikzpicture}
};
\node[draw=yellow,fill=yellow!20,inner sep=-1.4,circle] (b2) at (-1,0) {
\begin{tikzpicture}
\node at (0:0) {$\circ$};
\node at (0:.3) {$\bullet$};
\node at (90:.3) {$\bullet$};
\node at (180:.3) {$\circ$};
\node at (270:.3) {$\bullet$};
\end{tikzpicture}
};
\node[inner sep=-1.4,circle] (b) at (0,0) {
\begin{tikzpicture}
\node at (0:0) {$\bullet$};
\node at (0:.3) {$\circ$};
\node at (90:.3) {$\bullet$};
\node at (180:.3) {$\circ$};
\node at (270:.3) {$\bullet$};
\end{tikzpicture}
};
\node[draw=yellow,fill=yellow!20,inner sep=-1.4,circle] (b3) at (1,0) {
\begin{tikzpicture}
\node at (0:0) {$\bullet$};
\node at (0:.3) {$\circ$};
\node at (90:.3) {$\bullet$};
\node at (180:.3) {$\bullet$};
\node at (270:.3) {$\circ$};
\end{tikzpicture}
};
\node[draw=yellow,fill=yellow!20,inner sep=-1.4,circle] (b4) at (2,0) {
\begin{tikzpicture}
\node at (0:0) {$\bullet$};
\node at (0:.3) {$\bullet$};
\node at (90:.3) {$\bullet$};
\node at (180:.3) {$\circ$};
\node at (270:.3) {$\circ$};
\end{tikzpicture}
};
\node[draw=yellow,fill=yellow!20,inner sep=-1.5,circle] (T1) at (-.5,-1) {
\begin{tikzpicture}
\node at (0:0) {$\bullet$};
\node at (0:.3) {$\circ$};
\node at (90:.3) {$\bullet$};
\node at (180:.3) {$\bullet$};
\node at (270:.3) {$\bullet$};
\end{tikzpicture}
};
\node[draw=yellow,fill=yellow!20,inner sep=-1.5,circle] (T2) at (.5,-1) {
\begin{tikzpicture}
\node at (0:0) {$\bullet$};
\node at (0:.3) {$\bullet$};
\node at (90:.3) {$\bullet$};
\node at (180:.3) {$\circ$};
\node at (270:.3) {$\bullet$};
\end{tikzpicture}
};
\draw[->] (t) -- (m1);
\draw[->] (t) -- (m2);  
\draw[->] (t) -- (m3);
\draw[->] (t) -- (m4);
\draw[->] (m2) -- (b);
\draw[->] (m2) -- (b1);
\draw[->] (m2) -- (b2);
\draw[->] (m3) -- (b3);
\draw[->] (m3) -- (b4);
\draw[->] (m3) -- (b);
\draw[->] (b) -- (T1);
\draw[->] (b) -- (T2);
\end{tikzpicture}
\hfil  
\begin{tikzpicture}[xscale=1.4,yscale=1.2]
\node[inner sep=-1.5,circle] (t) at (0,2) {
\begin{tikzpicture}
\node at (0:0) {$\circ$};
\node at (0:.3) {$\circ$};
\node at (90:.3) {$\bullet$};
\node at (180:.3) {$\circ$};
\node at (270:.3) {$\circ$};
\end{tikzpicture}
};
\node[draw=yellow,fill=yellow!20,inner sep=1,circle] (m1) at (0,1) {
$\star$
};
\node[inner sep=-1.5,circle] (m2) at (-.8,1) {
\begin{tikzpicture}
\node at (0:0) {$\circ$};
\node at (0:.3) {$\circ$};
\node at (90:.3) {$\bullet$};
\node at (180:.3) {$\circ$};
\node at (270:.3) {$\bullet$};
\end{tikzpicture}
};
\node[inner sep=-1.5,circle] (m3) at (.8,1) {
\begin{tikzpicture}
\node at (0:0) {$\bullet$};
\node at (0:.3) {$\circ$};
\node at (90:.3) {$\bullet$};
\node at (180:.3) {$\circ$};
\node at (270:.3) {$\circ$};
\end{tikzpicture}
};
\node[draw=yellow,fill=yellow!20,inner sep=1,circle] (b2) at (-.8,0) {
$\star$
};
\node[inner sep=-1.4,circle] (b) at (0,0) {
\begin{tikzpicture}
\node at (0:0) {$\bullet$};
\node at (0:.3) {$\circ$};
\node at (90:.3) {$\bullet$};
\node at (180:.3) {$\circ$};
\node at (270:.3) {$\bullet$};
\end{tikzpicture}
};
\node[draw=yellow,fill=yellow!20,inner sep=1,circle] (b3) at (.8,0) {
$\star$
};
\node[draw=yellow,fill=yellow!20,inner sep=1,circle] (T) at (0,-1) {
$\star$
};
\node[inner sep=1,circle] () at (0,-1.3) {};

\draw[dashed,blue,->] (t) -- (m1);
\draw[->] (t) -- (m2);  
\draw[->] (t) -- (m3);
\draw[->] (m2) -- (b);
\draw[dashed,blue,->] (m2) -- (b2);
\draw[dashed,blue,->] (m3) -- (b3);
\draw[->] (m3) -- (b);
\draw[dashed,blue,->] (b) -- (T);
\end{tikzpicture}

\caption{
\label{fig:mapOrd}
Gamegraphs for $\TER_{\{v_i\}}(W_5)$ and $*1:\DNT_{\{v_i\}}(W_5)$.
}
\end{figure}
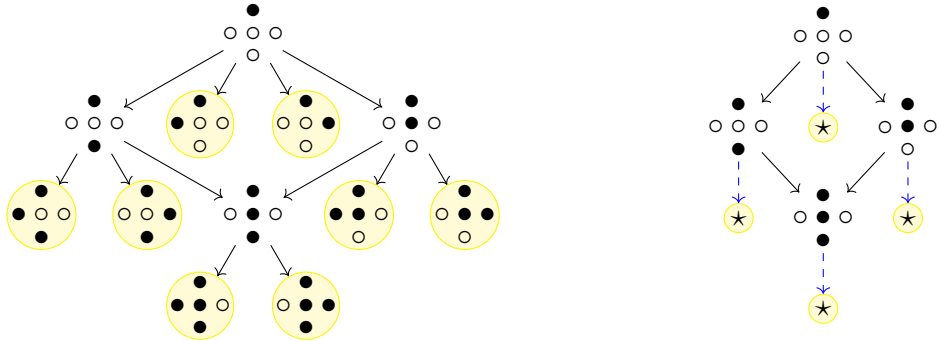

\begin{thm}
For wheel graphs,  
\[
\nim(\TER(W_n))=\begin{cases}
2, &  n \text{ mod } 34 \in \{5, 6,10,11,25,26,30,31\}\\
1, & \text{otherwise}.
\end{cases}
\]
\end{thm}

\begin{proof}
By symmetry, the positions with a single rim vertex selected are essentially the same. So $a:=\nim(\{v_i\})$ is independent of the choice of $i$. A position that contains a rim vertex $v_i$ has nonzero nim-value since the next player can win by marking $v_{i+1}$. Hence $a>0$. Also $\nim(\{c\})=0$ since each option $\{c,v_i\}$ of $\{c\}$ has nonzero nim-value. Thus
\[
\nim(\TER(W_n))=\mex(\nim(\{\{c\},\{v_1\},\ldots,\{v_n-1\}\}))=\mex(\{0,a\})\in\{1,2\},
\]
as depicted in Figure~\ref{fig:WheelTERCases}. In fact, $\nim(\TER(W_n))$ is $2$ if $a=1$ and $1$ if $a>1$.

By Propositions~\ref{prop:dawsonchessplusoneequivalence}, \ref{prop:AddWinningOption}, and~\ref{prop:sutjOP},
\[
a=\nim(\TER_{\{v_i\}}(W_n))=\nim(*1:\DNT_{\{v_i\}}(W_n))=\nim(\dawsonschess(n-4)+*1)+1.
\]
This value is $1$ exactly when $\nim(\dawsonschess(n-4))=1$. So the result follows from Propositions~\ref{prop:dawsonschess-sequence}.
\end{proof}

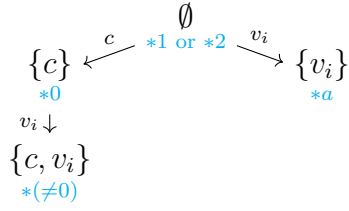
\begin{figure}[h]
\begin{tikzpicture}[xscale=1.8,yscale=1.3]
\node (0) at (0,-.4) {$\underset{\color{cyan}*1\text{ or } *2}{\emptyset}$};
\node (c) at (-1,-.9) {$\underset{\color{cyan}*0}{\{c\}}$};
\draw[->] (0) -- node[above ]{$\scriptstyle c$} (c);
\node (r) at (1,-.9) {$\underset{\color{cyan}*a}{\{v_i\}}$};
\draw[->] (0) -- node[above ]{$\scriptstyle v_i$} (r);
\node (cr) at (-1,-1.9) {$\underset{\color{cyan}*(\ne 0)}{\{c,v_i\}}$};
\draw[->] (c) -- node[left]{$\scriptstyle v_i$} (cr);
\end{tikzpicture}

\caption{
\label{fig:WheelTERCases}
Partial gamegraph for $\TER(W_n)$.
}
\end{figure}

\subsection{Generalized wheel graphs}
The \emph{generalized wheel graph} $W_{m,n}$ with $m\ge 2$ and $n\ge 3$ is the join $\overline{K}_m+C_n$ with $m+n$ total vertices in
$V(\overline{K}_m)=\{c_1, \ldots, c_m\}$ and $V(C_n)=\{v_1, \ldots, v_n\}$.

\begin{prop}\label{prop:GeneralizedWheel3}
For generalized wheel graphs, \[\nim(\DNT(W_{m,3}))=1,\quad \nim(\TER(W_{m,3}))=2.\] 
\end{prop}

\begin{proof}
Since $C_3$ is isomorphic to $K_3$, $W_{m,3}$ is a complete split graph. So the result follows from Proposition~\ref{prop:splitTER}.
\end{proof}

\begin{prop}\label{prop:max-genwheel}
For generalized wheel graphs with $n\ge 4$,
\[ 
\mathcal{N}^* = \{ \{c_i, c_j\}^c \mid i \ne j \} \cup
\{ \{v_k, v_\ell\}^c \mid  \text{$k\not=\ell$ and $v_k$, $v_\ell$ nonadjacent}\}. 
\]
\end{prop}

\begin{proof}
We know from \cite[Proposition 7.27]{BEMSS2024} that 
\[
\mathcal{G}^*=\complement(\mathcal{N})=\{\{c_i,v_j,v_{j+1}\}^c\mid i \in \{1,\ldots,m\}, j \in \{1,\ldots,n\}\}. 
\]
It is easy  to verify that
\[ 
\mathcal{G}=\Tr(\mathcal{G}^*))=
\{ \{c_i, c_j\} \mid i \ne j \} \cup
\{ \{v_k, v_\ell\} \mid  \text{$k\not=\ell$ and $v_k$, $v_\ell$ nonadjacent}\},
\]
which proves the claim about $\mathcal{N}^*=\complement(\mathcal{G})$.
\end{proof}

\begin{prop}
For generalized wheel graphs with $n\ge 4$, $\nim(\DNT(W_{m,n}))=\pty(V)$.
\end{prop}

\begin{proof}
The result follows from Propositions~\ref{prop:max-genwheel} and \ref{prop:all-same}.
\end{proof}

We now focus on $\TER(W_{m,n})$ for the rest of the section. We will map this game to another simpler game. To describe this map, we need to introduce some terminology. 

A  \emph{partition} $\lambda$ of a nonnegative integer $n$ is a multiset of nonnegative integers whose sum is $n$. It is customary to write the elements of a partition as a nonincreasing list $\lambda=[l_1,\ldots,l_k]$. We will sometimes use the notation $\lambda^n$ if we want to emphasize that we have a partition of $n$. We call each $\lambda_i$ a \emph{part} of $\lambda$. We say that we \emph{split} $\lambda$ if we replace a part $\alpha$ with a pair $\beta, \gamma \geq 0$ such that $\beta+\gamma=\alpha-1$. The special split when $\gamma=0$ is called a \emph{decrement}.

It will sometimes be useful to write $\lambda^n_o$ for a partition of $n$ with exactly $o$ odd parts.
Replacing an even part with an odd and an even part increases the number of odd parts by $1$. Replacing an odd part with two even parts decreases the number of odd parts by one, whereas replacing an odd part with two odd parts increases the number of odd parts by one. Hence a split of $\lambda^n_0$ always results in $\lambda^{n-1}_1$. A split of $\lambda^n_1$ results in either $\lambda^{n-1}_0$ or $\lambda^{n-1}_2$, with the former always being available.  

Once a rim vertex has been selected in a position $P$ of $\TER(W_{m,n})$, we define $f(P):=(c,\lambda)$, where $c$ is the number of unselected central vertices remaining and $\lambda$ is the partition that describes the sizes of the consecutive clusters of unselected rim vertices. If $P$ is a terminal position then $f(P)$ is $(0,[1])$, $(0,[2])$, $(1,[\,])$, $(1,[1])$, or $(1,[2])$. If a rim vertex has not yet been selected, we will use the notation $\Lambda^n$ to denote all of the original rim vertices, and we define $f(P_0):=(m,\Lambda^n)$ for the starting position $P_0$ of $\TER(W_{m,n})$. We consider $\Lambda^n$ to be a special partition whose only split is $[n-1]$. We allow $\lambda^n_{\pty(n)}$ to mean $\Lambda^n$.

Let $\SPL(m,n)$ be the game whose set of positions is $\{f(P)\mid P\in\TER(W_{m,n})\}$. 
There are two possible moves from position $(c,\lambda)$. A \emph{dehub} decreases $c$ by 1. The other option is a split of $\lambda$. Selecting a central vertex in position $P$ decreases the first component of $f(P)$, while selecting a rim vertex creates a split of the second component of $f(P)$. So $f:\TER(W_{m,n})\to\SPL(m,n)$ is an option-preserving map. As a consequence of \cite{Baltushkin,Basic}, $\nim(\TER(W_{m,n}))=\nim(\SPL(m,n))$.

\begin{example}
Figure~\ref{fig:fimage} shows a play in $\TER(W_{2,4})$ and its image in $\SPL(2,4)$ under the option-preserving map $f$.
\end{example}

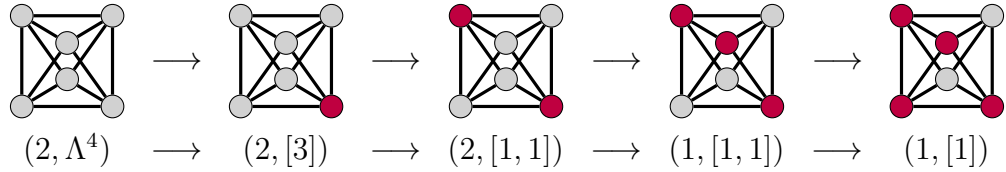
\begin{figure}[h]
\begin{tabular}{ccccccccc}
\begin{tikzpicture}[scale=1.2,auto,,baseline=.5cm]
\node (1) at (0,0) [small vert] {};
\node (2) at (1,0) [small vert] {};
\node (3) at (1,1) [small vert] {};
\node (4) at (0,1) [small vert] {};
\node (5) at (.5,.7) [small vert] {};
\node (6) at (0.5,.3) [small vert] {};
\path [b] (1) to (2) to (3) to (4) to (1);
\path [b] (1) to (5);
\path [b] (2) to (5);
\path [b] (3) to (5);
\path [b] (4) to (5);
\path [b] (1) to (6);
\path [b] (2) to (6);
\path [b] (3) to (6);
\path [b] (4) to (6);
\begin{pgfonlayer}{background}
\end{pgfonlayer}
\end{tikzpicture}
& $\longrightarrow$ &
\begin{tikzpicture}[scale=1.2,auto,,baseline=.5cm]
\node (1) at (0,0) [small vert] {};
\node (2) at (1,0) [small vert,fill=purple] {};
\node (3) at (1,1) [small vert] {};
\node (4) at (0,1) [small vert] {};
\node (5) at (.5,.7) [small vert] {};
\node (6) at (0.5,.3) [small vert] {};
\path [b] (1) to (2) to (3) to (4) to (1);
\path [b] (1) to (5);
\path [b] (2) to (5);
\path [b] (3) to (5);
\path [b] (4) to (5);
\path [b] (1) to (6);
\path [b] (2) to (6);
\path [b] (3) to (6);
\path [b] (4) to (6);
\begin{pgfonlayer}{background}
\end{pgfonlayer}
\end{tikzpicture}
& $\longrightarrow$ &
\begin{tikzpicture}[scale=1.2,auto,baseline=.5cm]
\node (1) at (0,0) [small vert] {};
\node (2) at (1,0) [small vert,fill=purple] {};
\node (3) at (1,1) [small vert] {};
\node (4) at (0,1) [small vert,fill=purple] {};
\node (5) at (.5,.7) [small vert] {};
\node (6) at (0.5,.3) [small vert] {};
\path [b] (1) to (2) to (3) to (4) to (1);
\path [b] (1) to (5);
\path [b] (2) to (5);
\path [b] (3) to (5);
\path [b] (4) to (5);
\path [b] (1) to (6);
\path [b] (2) to (6);
\path [b] (3) to (6);
\path [b] (4) to (6);
\begin{pgfonlayer}{background}
\end{pgfonlayer}
\end{tikzpicture}
& $\longrightarrow$ &
\begin{tikzpicture}[scale=1.2,auto,baseline=.5cm]
\node (1) at (0,0) [small vert] {};
\node (2) at (1,0) [small vert,fill=purple] {};
\node (3) at (1,1) [small vert] {};
\node (4) at (0,1) [small vert,fill=purple] {};
\node (5) at (.5,.7) [small vert,fill=purple] {};
\node (6) at (0.5,.3) [small vert] {};
\path [b] (1) to (2) to (3) to (4) to (1);
\path [b] (1) to (5);
\path [b] (2) to (5);
\path [b] (3) to (5);
\path [b] (4) to (5);
\path [b] (1) to (6);
\path [b] (2) to (6);
\path [b] (3) to (6);
\path [b] (4) to (6);
\begin{pgfonlayer}{background}
\end{pgfonlayer}
\end{tikzpicture}
& $\longrightarrow$ &
\begin{tikzpicture}[scale=1.2,auto,baseline=.5cm]
\node (1) at (0,0) [small vert,fill=purple] {};
\node (2) at (1,0) [small vert,fill=purple] {};
\node (3) at (1,1) [small vert] {};
\node (4) at (0,1) [small vert,fill=purple] {};
\node (5) at (.5,.7) [small vert,fill=purple] {};
\node (6) at (0.5,.3) [small vert] {};
\path [b] (1) to (2) to (3) to (4) to (1);
\path [b] (1) to (5);
\path [b] (2) to (5);
\path [b] (3) to (5);
\path [b] (4) to (5);
\path [b] (1) to (6);
\path [b] (2) to (6);
\path [b] (3) to (6);
\path [b] (4) to (6);
\end{tikzpicture}\\
\vspace{-3mm}
 & & & & & & & & \\
$(2,\Lambda^4)$ & $\longrightarrow$ & $(2,[3])$ & $\longrightarrow$ & $(2,[1,1])$ & $\longrightarrow$ & $(1,[1,1])$ & $\longrightarrow$ & $(1,[1])$
\end{tabular}

\caption{
\label{fig:fimage}
A play in $\TER(W_{2,4})$ and its image in $\SPL(2,4)$.
}
\end{figure}

\begin{rem}
As in~\cite{impartialremovinggamesgrid}, we are going to describe complicated winning strategies for the second player using \emph{case analysis diagrams}. A case analysis diagram is a digraph whose vertices are either single positions or sets of positions. Sets of positions are usually described using some parameters and conditions on these parameters. Single headed arrows represent the possible moves for the first player, while double headed arrows represent the winning replies for the second player. The diagram may contain cycles. The second player breaks out of these cycles along moves represented by dashed arrows when these moves become available. This is guaranteed since the game does not have infinite plays. A sink vertex is either a terminal position of the game or a nonterminal position that has already been proved to be a losing position. The latter is indicated by {\color{cyan}$*0$}  or a reference to another case analysis diagram. These sink vertices only have double headed incoming arrows. The starting position of the strategy is usually on the top of the diagram but a diagram can have several starting positions indicated by {\color{red}$\cc{i}$} symbols for some number $i$. Notation such as (X.$i$) indicates that the diagram continues at starting point $\cc{i}$ of Diagram~(X). Occasionally, a $\bullet$ indicates a position after the first player's move for which a description is not necessary.
\end{rem}

We demonstrate the use of case analysis diagrams with a familiar game.

\begin{example}
Figure~\ref{fig:nplusn} shows a case analysis diagram for a strategy of the second player to win $*n+*n$. Play starts at \cc{1} if $n>1$ and at \cc{2} if $n=0$. A dashed arrow becomes available for the second player when the first player moves to a position that makes $a=0$.
\end{example}

\begin{figure}[h]

\begin{tikzpicture}[xscale=3,yscale=.45]
\node (0) at (0,2) {$*b+*b$};
\node at (0.155) {$\cc{1}$};
\node (l) at (-1,1) {$*a+*b$};
\node (r) at (1,1) {$*b+*a$};
\node (b) at (0,0) {$*0+*0$};
\node at (b.155) {$\cc{2}$};
\draw[->] (0) to[bend right=15] (l);
\draw[->>] (l) to[bend right=15] (0);
\draw[->] (0) to[bend left=15] (r);
\draw[->>] (r) to[bend left=15] (0);
\draw[->>,dashed] (l) to (b);
\draw[->>,dashed] (r) to (b);
\end{tikzpicture}

\caption{
\label{fig:nplusn}
Case analysis diagram for a strategy with $0\le a<b$. 
}
\end{figure}
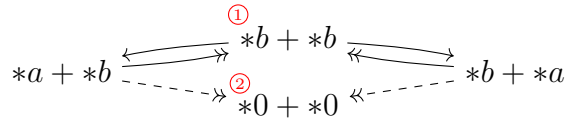

\begin{lemma}\label{lem:SPL33}
Position $(3,[3])$ of $\SPL$ has nim-value $1$.
\end{lemma}

\begin{proof}
The case analysis diagram in Figure~\ref{fig:Case3nnEndgame3} shows a winning strategy for the second player starting at $\cc{1}$.
\end{proof}

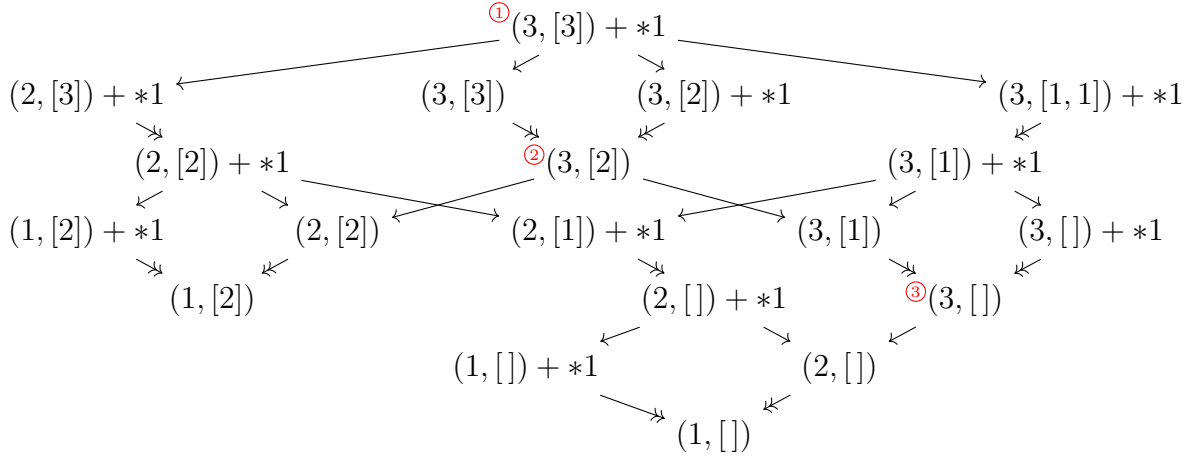
\begin{figure}[h]
\begin{tikzpicture}[yscale=0.45,xscale=0.83]
\node (01) at (2,6) {$(3,[3])+*1$};
\node at (01.170) {$\cc{1}$};
\node (11) at (-6,4) {$(2,[3])+*1$};
\node (12) at (10,4) {$(3,[1,1])+*1$};
\node (13) at (4,4) {$(3,[2])+*1$};
\node (14) at (0,4) {$(3,[3])$};
\draw[->] (01) -- (11);
\draw[->] (01) -- (12);
\draw[->] (01) -- (13);
\draw[->] (01) -- (14);
\node (21) at (-4,2) {$(2,[2])+*1$};
\node (22) at (8,2) {$(3,[1])+*1$};
\node (24) at (2,2) {$(3,[2])$};
\node at (24.170) {$\cc{2}$};
\draw[->>] (11) -- (21);
\draw[->>] (12) -- (22);
\draw[->>] (13) -- (24);
\draw[->>] (14) -- (24);
\node (31) at (-6,0) {$(1,[2])+*1$};
\node (32) at (2,0) {$(2,[1])+*1$};
\node (33) at (-2,0) {$(2,[2])$};
\node (34) at (10,0) {$(3,[\,])+*1$};
\node (35) at (6,0) {$(3,[1])$};
\draw[->] (21) -- (31);
\draw[->] (21) -- (32);
\draw[->] (21) -- (33);
\draw[->] (22) -- (32);
\draw[->] (22) -- (34);
\draw[->] (22) -- (35);
\draw[->] (24) -- (33);
\draw[->] (24) -- (35);
\node (41) at (-4,-2) {$(1,[2])$};
\node (42) at (4,-2) {$(2,[\,])+*1$};
\node (43) at (8,-2) {$(3,[\,])$};
\node at (43.170) {$\cc{3}$};
\draw[->>] (31) -- (41);
\draw[->>] (32) -- (42);
\draw[->>] (33) -- (41);
\draw[->>] (34) -- (43);
\draw[->>] (35) -- (43);
\node (51) at (1,-4) {$(1,[\,])+*1$};
\node (52) at (6,-4) {$(2,[\,])$};
\draw[->] (42) -- (51);
\draw[->] (42) -- (52);
\draw[->] (43) -- (52);
\node (63) at (4,-6) {$(1,[\,])$};
\draw[->>] (51) -- (63);
\draw[->>] (52) -- (63);
\end{tikzpicture}

\caption{
\label{fig:Case3nnEndgame3}
Case analysis diagram (A) for one possible endgame. The alternate starting positions $\cc{2}$ and $\cc{3}$ will be needed later in the paper.
}
\end{figure}

\begin{lemma}
\label{lemma:AllOnes}
A position $(c,[1^c])$ of $\SPL$ with $c\ge 1$ has nim-value $0$.
\end{lemma}

\begin{proof}
The second player can win by splitting whenever the first player decreases and decreasing whenever the first player splits.
\end{proof}

\begin{lemma}
\label{lemma:GeneralizedWheel2n}
A position $(c,\lambda^{2l}_0)$ of $\SPL$ with $c\in\{0,2\}$ and $l\geq 2$ has nim-value $0$.
\end{lemma}

\begin{proof}
A winning strategy for the second player is described in Figure~\ref{fig:EvenLemma}.  For $l=2$, the strategy starts at $(0,[4])$, $(0,[2,2])$, $(2,[4])$, or $(2,[2,2])$.  For $l \geq 3$, the starting position is at \cc{1} or \cc{2}. Note that the possible partitions described as $\lambda_1^5$ are $[5]$, $[4,1]$, $[3,2]$, and $[2,2,1]$. From each of these either $[4]$ or $[2,2]$ is available after a split, indicated by a dashed arrow in the diagram. 
\end{proof}

\begin{figure}[h]
\begin{tikzpicture}[yscale=0.49,xscale=0.76]
\node (0) at (0,2.5) {${(2,\lambda^{2k}_0)}$};
\node[draw,circle,red,inner sep=.7] at (0.170) {$\scriptscriptstyle 1$};
\node (a1) at (-4,0) {$(1,\lambda^{2k}_0)$};
\node (a2) at (6,2) {$(2,\lambda^{2k-1}_1)$};
\draw[->,blue] (0) -- (a1);
\draw[->,orange] (0) to[bend left=10] (a2);
\node (b1) at (-7,-2) {$(0,\lambda^{2k}_0)$};
\node[draw,circle,red,inner sep=.7] at (b1.170) {$\scriptscriptstyle 2$};
\draw[->>,blue] (a1) -- (b1);
\draw[->>,orange] (a2) to[bend left=10] (0);
\node (c1) at (-3,-2.5) {$(0,\lambda^{2k-1}_1)$};
\node (c2) at (3,0) {$(2,4)$};
\node[draw,circle,red,inner sep=.7] at (c2.170) {$\scriptscriptstyle 3$};
\node (c3) at (9.5,0) {$(2,[2,2])$};
\node[draw,circle,red,inner sep=.7] at (c3.170) {$\scriptscriptstyle 4$};
\draw[->,orange] (b1) to[bend left=10] (c1);
\draw[dashed,->>,orange] (a2) -- (c2);
\draw[dashed,->>,orange] (a2) -- (c3);
\node (d2) at (1,-2) {$(1,4)$};
\node (d3) at (4,-2) {$(2,[3])$};
\node (d4) at (7,-2) {$(2,[2,1])$};
\node (d5) at (11,-2) {$(1,[2,2])$};
\draw[->>,orange] (c1) to[bend left=10] (b1);
\draw[->,blue] (c2) -- (d2);
\draw[->,orange] (c2) -- (d3);
\draw[->,orange] (c2) -- (d4);
\draw[->,orange] (c3) -- (d4);
\draw[->,blue] (c3) -- (d5);
\node (e1) at (-5,-5) {$(0,4)$};
\node[draw,circle,red,inner sep=.7] at (e1.170) {$\scriptscriptstyle 5$};
\node (e2) at (-1,-5) {$(0,[2,2])$};
\node[draw,circle,red,inner sep=.7] at (e2.170) {$\scriptscriptstyle 6$};
\node (e3) at (5.5,-5) {$(2,[1,1])$};
\draw[dashed,->>,orange] (c1) -- (e1);
\draw[dashed,->>,orange] (c1) -- (e2);
\draw[->>,blue] (d2) -- (e1);
\draw[->>,blue] (d5) -- (e2);
\draw[->>,orange] (d3) -- (e3);
\draw[->>,orange] (d4) -- (e3);
\node (f1) at (-5,-7) {$(0,[3])$};
\node (f2) at (-1,-7) {$(0,[2,1])$};
\node (f3) at (3.5,-7) {$(1,[1,1])$};
\node (f4) at (7.5,-7) {$(2,[1])$};
\draw[->,orange] (e1) -- (f1);
\draw[->,orange] (e1) -- (f2);
\draw[->,orange] (e2) -- (f2);
\draw[->,blue] (e3) -- (f3);
\draw[->,orange] (e3) -- (f4);
\node (g1) at (-3,-9) {$(0,[2])$};
\node (g2) at (5.5,-9) {$(1,[1])$};
\draw[->>,orange] (f1) -- (g1);
\draw[->>,orange] (f2) -- (g1);
\draw[->>,orange] (f3) -- (g2);
\draw[->>,blue] (f4) -- (g2);
\end{tikzpicture}

\caption{
\label{fig:EvenLemma}
Case analysis diagram (B) for a strategy  with $k\ge 3$.  Here, $(i,4)$ stands for either $(i,\Lambda^4)$ or $(i,[4])$ with $i \in \{0,1,2\}$. Arrow colors distinguish between dehub and split moves.
}
\end{figure}
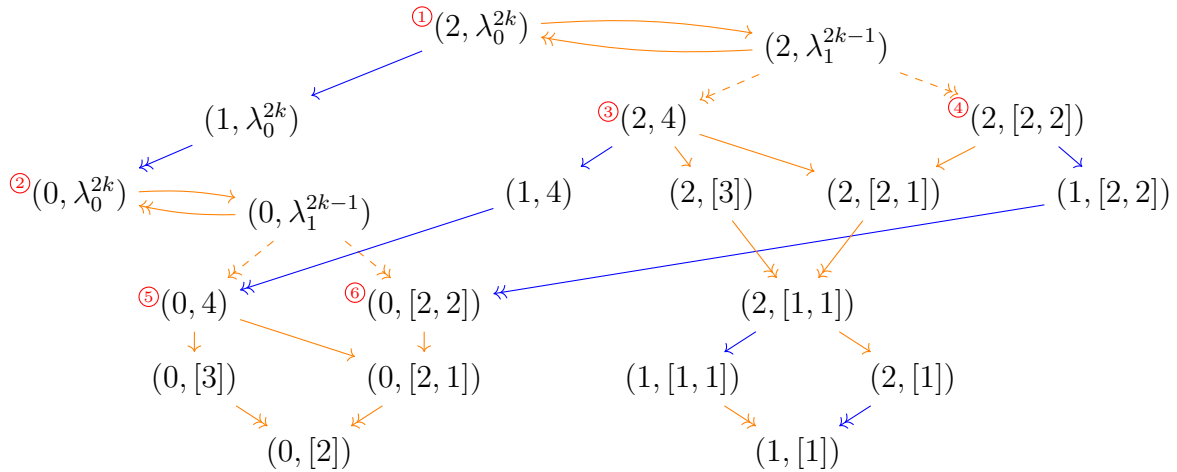

We say that a partition $\lambda:=[l_1,l_2,\ldots,l_k]$ is \emph{solid} if $l_1 \geq 4$ or $l_2 \geq 2$. We will write $\bm{\lambda}$ to indicate a solid partition. If $\lambda$ is not solid, then $\lambda$  must have the form $[1^a]$, $[2,1^a]$, or $[3,1^a]$ for some $a \geq 0$. It is impossible to split a solid partition into a partition of the form $[1^a]$. We will also consider $\Lambda^n$ to be solid for all $n \geq 5$.

\begin{lemma}
\label{lemma:SolidPartitions}
A solid partition $\bm{\lambda}$ different from $[4]$ and $[2,2]$ can be split into a solid partition.
\end{lemma}

\begin{proof}
If $\bm{\lambda}=\Lambda^n$, then $n\ge 5$ and so the only split $[n-1]$ is solid.
Thus, we may assume that $\bm{\lambda}=[l_1,l_2,\ldots,l_k]$. If $l_1>4$, then we can decrement $l_1$. If $l_1=4$, then $l_2\in\{1,2,3,4\}$ since $\lambda \not= [4]$, so we can decrement $l_2$. If $l_1=3$, then $l_2\in\{2,3\}$, so we can decrement $l_1$. Finally, if $l_1=2$, then $l_2=2$ and $l_3 >0$ since $\lambda \not=[2,2]$, so we can decrement $l_3$.
\end{proof}

\begin{lemma}
\label{lem:solidIfBig}
A partition $\lambda_o^c$ with $o\le 2$ and $c\ge 5$ is solid.
\end{lemma}

\begin{proof}
For a contradiction suppose that $\lambda_o^c$ is not solid. Then $\lambda_o^c$ must be $[1^a]$, $[2,1^a]$, or $[3,1^a]$ for some $a$. Since $c\ge 5$, we must have $a\ge 2$. This implies the contradiction $o\ge 3$.
\end{proof}

\begin{lemma}
\label{lem:solidStrat}
A position $(c,\bm{\lambda}^{c+2})$ of $\SPL$ with $c\ge 3$ has nim-value $0$.
\end{lemma}

\begin{proof}
A winning strategy for the second player is described in Figure~\ref{fig:GenWheelPhase2}. The dotted arrow is justified by Lemma~\ref{lemma:SolidPartitions}. We also use the fact that it is impossible to split a solid partition to $[1^a]$. Positions $(2,[4])$ and $(2,[2,2])$ are losing by Lemma~\ref{lemma:GeneralizedWheel2n}. Position $(c,[1^c])$ is losing by Lemma~\ref{lemma:AllOnes}. Note that the strategy works even if $\lambda_0^{c+2}=\Lambda^{c+2}$.
\end{proof}

\begin{figure}[h]

\begin{tikzpicture}[xscale=0.7,yscale=.66]
\node (0) at (0,6) {$(c,\bm{\lambda}^{c+2})$};
\node[red] at (0.170) {$\cc{1}$};
\node (a1) at (-6,4) {$(c-1,\bm{\lambda}^{c+2})$};
\node (a2) at (-2,4) {$(c,\bm{\lambda}^{c+1})$};
\node (a3) at (2,4) {$(c,[2,1^{c-1}])$};
\node (a4) at (6,4) {$(c,[3,1^{c-2}])$};
\draw[->,blue] (0.190) -- (a1);
\draw[densely dotted,->>,orange] (a1.12) -- (0);
\draw[->,orange] (0) to[bend right=7] (a2);
\draw[->>,blue] (a2) to[bend right=7] (0);
\draw[->,orange] (0) -- (a3);
\draw[->,orange] (0) -- (a4);
\node (b2) at (6,2) {$\underset{\color{cyan}{*0}}{(c,[1^c])}$};
\draw[->>,orange] (a3) -- (b2);
\draw[->>,orange] (a4) -- (b2);
\node (c1) at (-6,2) {$(2,[4])$};
\node at (c1.-90) {\tiny (B.3)};
\node (c2) at (-2,2) {$(2,[2,2])$};
\node at (c2.-90) {\tiny (B.4)};
\draw[dashed,->>,orange] (a1) -- (c1);
\draw[dashed,->>,orange] (a1) -- (c2);
\draw[dashed,->>,blue] (a2) -- (c1);
\draw[dashed,->>,blue] (a2) -- (c2);
\end{tikzpicture}

\caption{
\label{fig:GenWheelPhase2}
Case analysis diagram (C) for a strategy with $c \geq 3$. 
}
\end{figure}
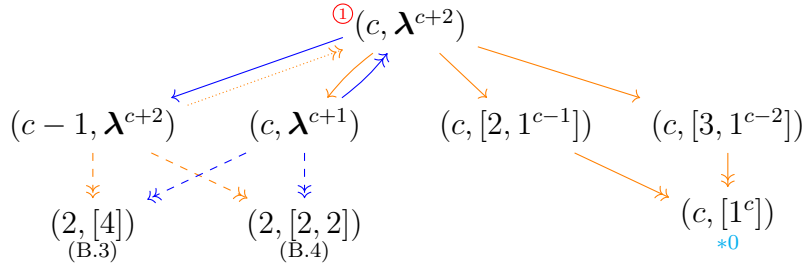

\begin{lemma}
\label{lem:SmallPartition}
A position $(c+(2k+3),\lambda^{c})$ of $\SPL$ with $c, k\ge 0$ has nim-value $0$. 
\end{lemma}

\begin{proof}
The second player can win by splitting whenever possible until position $(3,[\,])$ is reached, as shown in Figure~\ref{fig.split}. This is a losing position as shown in Figure~\ref{fig:Case3nnEndgame3}.
\end{proof}

\begin{figure}[h]

\begin{tikzpicture}[xscale=3,yscale=.5]
\node (0) at (0,2) {$(c+2k+3,\lambda^c)$};
\node at (0.170) {$\cc{1}$};
\node (l) at (1,2) {$\bullet$};
\node (b) at (0.4,0) {$(3,[\,])$};
\node at (b.150) {$\cc{2}$};
\node at (b.-90) {\tiny (A.3)};
\draw[->] (0) to[bend left=12] (l.150);
\draw[->>] (l.210) to[bend left=12] (0);
\draw[->>,dashed] (l.-100) to (b);
\end{tikzpicture}

\caption{
\label{fig.split}
Case analysis diagram (D) for a strategy with $c,k\ge 0$. 
}
\end{figure}

\begin{thm}
For generalized wheel graphs with $n \geq 4$, 
\[
\nim(\TER(W_{m,n}))=\begin{cases}
0, & n+m \text{ even and } 2\le d\\
0, & n+m \text{ odd and } d \leq 0\\
1, & n+m \text{ even and } d \leq 0\\
1, & n+m \text{ odd and } 2\le d\\
2, & d=1,
\end{cases}
\]
where $d:=n-m$.
\end{thm}

\begin{proof}
It is sufficient to prove the corresponding result for $\SPL(m,n)$ due to the option-preserving map $f$. We consider the five cases of our formula separately. In each case we show that the second player can win.


{\bf Case 1.} $\pty(n+m)=0$ and $m+2\le n$:
The second player's goal is to keep the unselected central and rim vertices even in a careful way until there are only two central vertices left or there are exactly two more rim than central vertices. Figure~\ref{fig:Case1Phase1BeginningNew} shows a case analysis diagram for a winning strategy. Position $(2,\lambda_0^{2+2k})$ is losing by Lemma~\ref{lemma:GeneralizedWheel2n}. Lemma~\ref{lem:solidIfBig} implies that ${\bm\lambda_0^{e+2}}$ is solid, so $(e,{\bm\lambda}_0^{e+2})$ is losing by Lemma~\ref{lem:solidStrat}. 

\begin{figure}[h]
\begin{tikzpicture}[yscale=0.2,xscale=0.67]
\node (00) at (-17,21) {$(o,\Lambda^{o+2k}_{1})$};
\node at (00.170) {$\cc{1}$};
\node (a1) at (-23,18) {$(o-1,\Lambda^{o+2k})$};
\node (a2) at (-12,18) {$(o,\lambda_0^{e+2k})$};
\node (a4) at (-17,10) {$(e,\lambda_0^{e+2l})$};
\node at (a4.160) {$\cc{3}$};

\node (b2) at (-17,3) {$(e,\lambda_0^{e+2l-1})$};
\node (c1) at (-17,16) {$(2,\lambda^{2+2k}_0)$};
\node at (c1.160) {$\cc{2}$};
\node at (c1.-90) {\tiny (B)};
\node (c2) at (-17,-4) {$(e,{\bm\lambda}_0^{e+2})$};
\node at (c2.170) {$\cc{4}$};
\node at (c2.-90) {\tiny (C.1)};

\draw[->,blue] (00) -- (a1);
\draw[->,orange] (00) -- (a2);
\draw[->>,orange] (a1) to[bend right=30] (a4);

\draw[->,blue] (a4) to[bend right=6] (a2);
\draw[->>,blue] (a2) to[bend right=6] (a4);
\draw[->,orange] (a4) to[bend left=4] (b2);
\draw[->>,orange] (b2) to[bend left=4] (a4);

\draw[->>,dashed,orange] (b2) -- (c2);
\draw[->>,dashed,orange] (a1) -- (c1);
\draw[->>,dashed,orange] (a1) to[bend right=15] (c2);
\draw[->>,dashed,blue] (a2) to (c1);
\draw[->>,dashed,blue] (a2) to[bend left=15] (c2);
\end{tikzpicture}

\caption{
\label{fig:Case1Phase1BeginningNew}
Case analysis diagram~(E) for Case~1 with odd $o \geq 3$, even $e \geq 4$, $l \geq 2$, and $k \geq 1$.
}
\end{figure}
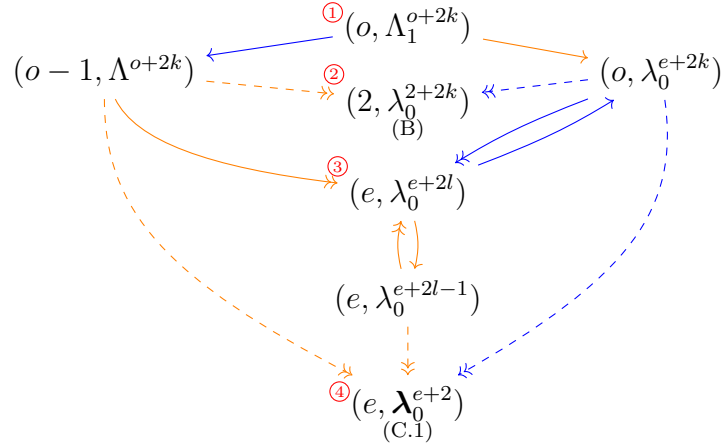


{\bf Case 2.} $\pty(n+m)=1$ and $m>n$: 
If $m\geq n+3$, then the second player wins by Lemma~\ref{lem:SmallPartition}. Otherwise, $m=n+1$ and the second player wins using the strategy shown in Figure~\ref{fig:Case2}.

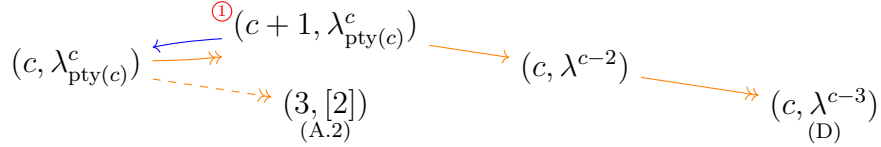
\begin{figure}[h]

\begin{tikzpicture}[yscale=.5,xscale=3.3]
\node (b1) at (1,2) {$(c+1,\lambda_{\pty(c)}^{c})$};
\node at (b1.170) {$\cc{1}$};
\node (b2) at (3,0) {$(c,\lambda^{c-3})$};
\node at (b2.-90) {\tiny (D)};
\node (c1) at (0,1) {$(c,\lambda_{\pty(c)}^{c})$};
\node (c2) at (2,1) {$(c,\lambda^{c-2})$};
\draw[->,orange] (b1) -- (c2);
\draw[->,blue] (b1) to[bend right=15] (c1);
\draw[->>,orange] (c1) to[bend right=15] (b1);
\draw[->>,orange] (c2) -- (b2);
\node (e1) at (1,0) {$(3,[2])$};
\node at (e1.-90) {\tiny (A.2)};
\draw[dashed,->>,orange] (c1) -- (e1);
\end{tikzpicture}

\caption{
\label{fig:Case2}
Case analysis diagram (F) for the endgame of Case 2 with $c\ge 3$.
}
\end{figure}


{\bf Case 3.} $\pty(n+m)=0$ and $m\geq n \geq 4$: 
We show that the second player can win $\SPL(m,n)+*1$.   If the first player splits or makes a move in $*1$, then the second player can move to $(m,[n-1])$.  Since $m \geq n-1$ and $m+(n-1)$ is odd, this is a losing position by Case~2. 

So we may assume that the first player initially moves to $(m-1,\Lambda^n)+*1$.  If $m \geq n+2$, then the second player can move to $(m-1,\Lambda^n)$ and follow the strategy from Case~2.  

So we may also assume that $m=n$ and the first player moves to $(n-1,\Lambda^n)+*1$.  Then the second player can move to $(n-1,[n-1])+*1$ and win following the strategy in Figure~\ref{fig:Case3nnEndgame2}.
The continuation from $(3,[3])+*1$ is shown in Figure~\ref{fig:Case3nnEndgame3}. Positions $(o,[o-1])$, $(e,\lambda^{e-1}_1)$, and $(o,\lambda^{o-3})$ are losing as shown in Figure~\ref{fig:Case2}.

\begin{figure}[h]
\begin{tikzpicture}[yscale=0.37,xscale=0.75]
\node (01) at (0,8) {$(o,[o])+*1$};
\node[draw,circle,red,inner sep=.7] at (01.170) {$\scriptscriptstyle 1$};
\node (11) at (-3,5) {$(o-1,[o])+*1$};
\node (12) at (0.5,5) {$(o,[o])$};
\node (13) at (4,5) {$(o,\lambda^{o-1}_0)+*1$};
\node (14) at (8,5) {$(o,\lambda^{o-1}_2)+*1$};
\draw[->] (01) -- (11);
\draw[->] (01) -- (12);
\draw[->] (01) -- (13);
\draw[->] (01) -- (14);
\node (21) at (-6,2) {$(e,[e])+*1$};
\node[draw,circle,red,inner sep=.7] at (21.170) {$\scriptscriptstyle 2$};
\node (22) at (0.5,2) {$(o,[o-1])$};
\node at (22.-90) {\tiny (F.1)};
\node (23) at (4,2) {$(o,\lambda^{o-1}_o)$};
\node (24) at (8,2) {$(o,\lambda^{o-1}_2)$};
\draw[->>] (11) -- (21);
\draw[->>] (12) -- (22);
\draw[->>] (13) -- (23);
\draw[->>] (14) -- (24);
\node (31) at (-9.5,-2) {$(e-1,[e])+*1$};
\node (32) at (-5,-2) {$(e,[e])$};
\node (33) at (-1,-2) {$(e,\lambda^{e-1}_1)+*1$};
\node (34) at (3,-2) {$(e,\lambda^{e}_o)$};
\node (35) at (6,-2) {$(e,\lambda^{e}_2)$};
\node (36) at (9,-2) {$(o,\lambda^{o-2})$};
\draw[->] (21) -- (31);
\draw[->] (21) -- (32);
\draw[->] (21) -- (33);
\draw[->] (23) -- (34);
\draw[->] (23) -- (36);
\draw[->] (23) -- (36);
\draw[->] (24) -- (35);
\draw[->] (24) -- (36);
\draw[->>] (31) to[bend left=40] (01);
\node (41) at (-9.5,-6) {$(3,[3])+*1$};
\node at (41.-90) {\tiny (A.1)};
\node (42) at (-1,-6) {$(e,\lambda^{e-1}_1)$};
\node at (42.-90) {\tiny (F.1)};
\node (43) at (9,-6) {$(o,\lambda^{o-3})$};
\node at (43.-90) {\tiny (D.1)};
\draw[dashed,->>] (31) -- (41);
\draw[->>] (32) -- (42);
\draw[->>] (33) -- (42);
\draw[->>] (34) -- (42);
\draw[->>] (35) -- (42);
\draw[->>] (36) -- (43);
\end{tikzpicture}

\caption{
\label{fig:Case3nnEndgame2}
Case analysis diagram (G) for the endgame of Case 3 starting at either $(o,[o])+*1$ or $(e,[e])+*1$ for some odd $o \geq 5$ or even $e \geq 4$.  
}
\end{figure}
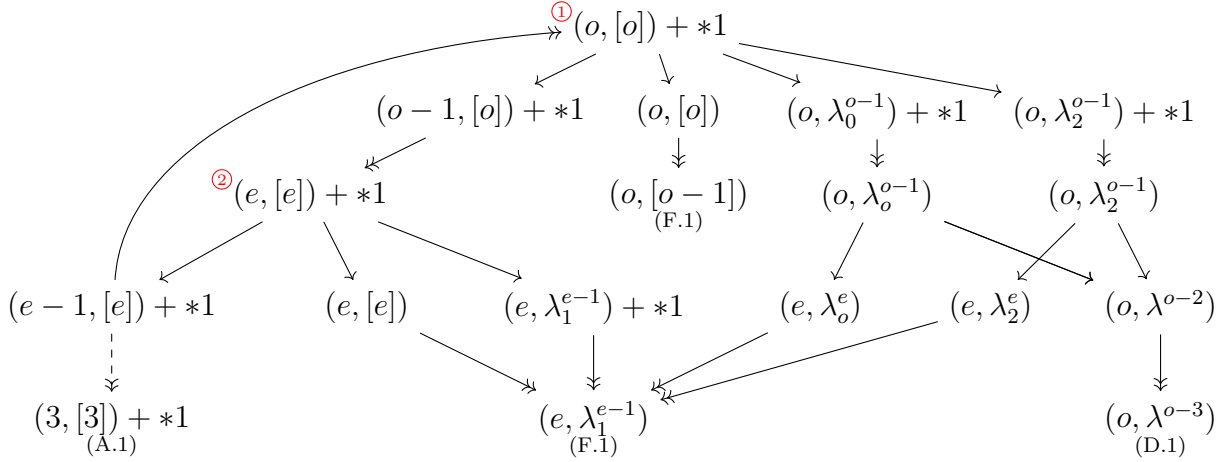

{\bf Case 4.}  $\pty(n+m)=1$ and $m+2 \leq n$:  
We will show that the second player can win $\SPL(m,n)+*1$. First suppose that $m=2$ and the first player initially moves to $(1,\Lambda^n)+*1$.  The second player then can move in $*1$ to get to $(0,[n-1])$ at the fourth move of the game.  Since $n \geq m+2=4$ and $n$ is odd, $n-1$ is even and at least $4$. Hence the second player can win by Lemma~\ref{lemma:GeneralizedWheel2n}.

Now suppose that  $m>2$ or the first player's initial selection is a split.  The second player can move to either $(m-1,\Lambda^n)$ where $m-1 \geq 2$ or $(m,[n-1])$ where $m \geq 2$. The second player now wins as in Case~1 since $\pty(n+m-1)=0$.

{\bf Case 5.} $m+1=n$: 
The options of the starting position of $\SPL(m,m+1)$ are $(m-1,\Lambda^{m+1})$ and $(m,[m])$.  The former is isomorphic to $\SPL(m-1,m+1)$ and has nim-value $*0$ by Case~1 since $(m-1)+(m+1)=2m$ is even and $m \geq 3$.  The latter has nim-value $*1$ by Figure~\ref{fig:Case3nnEndgame2}.

Therefore, the nim-value for Case~5 is $2$.
\end{proof}

Recall that the case $W_{m,3}$ was handled in Proposition~\ref{prop:GeneralizedWheel3}.

\section{Further directions}\label{sec:Closing}

\begin{enumerate}
\item 
Determine whether Conjecture~\ref{conj:oddCyclic} is true.  The difficulty is that there is no obvious pairing strategy for the second player due to there being an odd number of vertices.
\item   
Sometimes the geodetic closure agrees with the convex hull, but not always. For example, they are not the same on $Q_n$ if $n\ge 3$.   What are the nim-values for the variations using geodetic closures instead of convex hulls?  
\item
The nim-value of an arbitrary hypergraph game can be any integer as shown in~\cite{SiebenHypergraph}. The available nim-values for generating groups is much more limited \cite{BeneshErnstSiebenGENSpectrum}. What is the spectrum of nim-values for each of the convex hull removing games?
\item 
One can study the convex hull hypergraph games on other finite combinatorial objects where there is a natural notion of geodesic.  The objects include hypergraphs, weighted graphs, and directed graphs such as the Hasse diagram of a poset and the Cayley digraph for a group.
\item
A case analysis diagram is essentially a quotient of a winning strategy, which is a pruned game digraph containing only the winning moves of the winner. It might be beneficial to develop a formal theory of these quotients similar to~\cite{dMorphism,Baltushkin,Basic}. 
\end{enumerate}

\section*{Acknowledgments}

This material is based upon work supported by the National Science Foundation under
Grant No.~DMS-1929284 while the authors were in residence at the Institute for Computational and Experimental Research in Mathematics (ICERM) in Providence, RI, via the Collaborate@ICERM program.

\bibliographystyle{plain}
\bibliography{game}

\newpage

\end{document}